\documentclass[a4paper,reqno,11pt]{amsart}

\usepackage{latexsym, amssymb, enumerate, amsmath}

\usepackage[pdftex]{graphicx}
\usepackage{lscape}
\usepackage{amsfonts}
\usepackage{amssymb}
\usepackage{fancyhdr}
\usepackage{setspace}
\usepackage{color}
\usepackage{enumitem}
\usepackage{float}

\usepackage{graphicx}
\usepackage{caption}
\usepackage{subcaption}

\newcommand{\ve}{^{\varepsilon}}
\newcommand{\ds}{\displaystyle}
\newcommand{\vr}{\varepsilon}

\newtheorem{thm}{Theorem}[section]
\newtheorem{defn}[thm]{Definition}

\begin{document}

\title[Comparison of hyperbolic and  parabolic chemotaxis models]{A numerical comparison between degenerate parabolic and quasilinear hyperbolic models of cell movements under chemotaxis}

\author[M. Twarogowska, R. Natalini, and M. Ribot]{M. Twarogowska$^1$ \and R. Natalini$^1$ \and M. Ribot$^2$ }

\thanks{ {\emph{keywords and phrases: }}chemotaxis, quasilinear hyperbolic problems with source, degenerate parabolic problems, comparison between parabolic and hyperbolic models, stationary solutions with vacuum, well-balanced scheme, asymptotic behavior}

\thanks{\vspace{0.2cm}          
$^1$ Istituto per le Applicazioni del Calcolo ``Mauro Picone'', 
Consiglio Nazionale delle Ricerche,  via dei Taurini 19, I-00185 Roma, Italy
({\tt mtwarogowska@gmail.com}, {\tt roberto.natalini@cnr.it}).}

\thanks{$^2$ Laboratoire J. A. Dieudonn\'e, UMR CNRS 7351, Universit\'e de Nice-Sophia Antipolis,
 Parc Valrose, F-06108 Nice Cedex 02, France
\&  Project Team COFFEE, INRIA Sophia Antipolis, France
 ({\tt ribot@unice.fr}).}

\begin{abstract}
We consider two models which were both designed to describe the movement of eukaryotic cells responding to chemical signals. Besides a common standard parabolic equation for the diffusion of a chemoattractant, like chemokines or growth factors, the two models differ for the equations describing the movement of cells. The first model is based on a quasilinear hyperbolic system with damping, the other one on a degenerate parabolic equation. The two models have the same stationary solutions, which may contain some regions with vacuum. We first explain in details how to discretize the quasilinear hyperbolic system through an upwinding technique, which uses an adapted reconstruction, which is able to deal with the transitions to vacuum. Then we concentrate on the analysis of asymptotic preserving properties of the scheme towards a discretization of the parabolic equation, obtained in the large time and large damping limit, in order to  present a numerical comparison between the asymptotic behavior of these two models. Finally we perform an accurate numerical comparison of the two models in the time asymptotic regime, which shows that the respective solutions have a quite different behavior for large times. 
\end{abstract}

\maketitle


\section{Introduction}\label{sec:introduction}

The movement of cells, bacteria or  other microorganisms under the effect of a chemical stimulus, represented by a chemoattractant, such as chemokines or growth factors, has been widely studied in mathematics in the last two decades, see \cite{H, M1,M2, Pe}, and various models involving partial differential equations have been proposed to describe this evolution. 
The basic unknowns in these chemotactic models are the density of individuals and the concentrations of some chemical attractants. One of the most considered models is the Patlak-Keller-Segel system \cite{KS,Patlak}, where the evolution of the density of cells is described by a parabolic equation, and the concentration of  a chemoattractant is generally given by a parabolic or elliptic equation, depending on the different regimes to be described and on the authors' taste. The behavior of this systems is quite well known now, at least for linear diffusions: in the one-dimensional case, the solution is always global in time \cite{NY07}, while in two and more dimensions the solutions exist globally in time
or blow up according to the size of the initial data, see \cite{CC08,CCE12} and references therein, and see the recent result of global existence for large data in the parabolic-parabolic case \cite{Biler}. However, a drawback of this model is that the diffusion leads alternatively to a fast dissipation or an explosive behavior, while in general, from a biological point of view, it is much more interesting to observe  the creation of patterns and permanent structures. 
 
In order to avoid these drawbacks and to improve the accuracy of the transient description, some modifications  of the original Keller-Segel formulation were introduced to prevent overcrowding, by taking into account the volume filling effect; see \cite{H,Pe,HP2001,PH2002}. For instance, in \cite{Kowalczyk2005, CCA06}, a  nonlinear diffusion is considered. More precisely, denoting by  $\rho(x,t)$  the density of cells and by $\phi(x,t)$  the concentration of a generic chemoattractant, the Keller-Segel-like system with nonlinear diffusion reads
\begin{equation}\label{eq:KellerSegel}
\left\{\begin{array}{ll}
\rho_{t}=P(\rho)_{xx}-\chi(\rho\phi_{x})_{x},\\
\delta \phi_{t}=D\phi_{xx}+a\rho-b\phi,
\end{array}\right.
\end{equation}
where $\chi, \, D, \, a$ and $b$ are given positive parameters.  The cells move following the direction of the gradient of the  concentration of chemoattractant with a response coefficient $\chi$; they also  diffuse and
 $P$ is  a phenomenological, density dependent function, which is usually given by a pressure law for isentropic gases, such as
\begin{equation}\label{eq:pressure_law}
P(\rho)=\kappa \rho^{\gamma},\qquad\gamma>1, \qquad  \kappa>0,
\end{equation}
which is intended to prevent the overcrowding of cells. Besides, the evolution of chemoattractant is still given by a linear diffusion equation with a source term which depends on $\rho$.
The chemoattractant is released by the cells, diffuses in the environment and it is  degraded in  finite time. The positive parameters $D,a,b$ are respectively its diffusion coefficient, the production rate, which is proportional to the cell density, and  the degradation rate. If $\delta=1$, we consider a parabolic-parabolic model and in the case where $\delta=0$, we deal with a parabolic-elliptic model.

Now, it is also  expected that a hyperbolic model will enable us to observe intermediate organized structures, like aggregation patterns, at a finer scale \cite{Perthame_survey}. In \cite{Dolak_Hillen2003, Hillen_Stevens2000} the advantage of the hyperbolic approach over the parabolic one was considered in the case of a semilinear model of chemotaxis based on the Cattaneo law. In particular, the authors described qualitatively  some experiments of patterns formation. Here, we focus on  a quasilinear hyperbolic model  of chemotaxis  introduced by Gamba et al. \cite{Gamba2} to describe the early stages of the vasculogenesis process.  
 This model  writes as a hyperbolic-parabolic system for the following unknowns: the density of cells $\rho(x,t)$,  their momentum  $\rho u(x,t)$ and  the concentration  $\phi(x,t)$ of a chemoattractant:
\begin{equation}\label{eq:main_system}
\left\{\begin{array}{l}
\displaystyle{\rho_{t}+(\rho u)_{x}=0,}\\
\displaystyle{(\rho u)_{t}+\left(\rho u^{2}+P(\rho)\right)_{x}=-\alpha\rho u+\chi\rho\phi_{x},}\\
\displaystyle{\phi_{t}=D\phi_{xx}+a\rho-b\phi}.
\end{array}\right.
\end{equation}
The positive constants $\chi$  and $\alpha$ measure respectively  the strength of the cells response to the concentration of the chemical substance and   the strength of the damping forces.  The pressure $P$ is still given by  the pressure law for isentropic gases \eqref{eq:pressure_law}.  This model of chemotaxis has been introduced to describe the results of in vitro experiments performed by Serini et al. \cite{Gamba1} using human endothelial cells which, randomly seeded on a matrigel, formed complex patterns with structures depending on the initial number of cells. 

 Although analytical results about this model are still far from being complete,  for the Cauchy problem on the whole space and in all space dimensions, so with no boundary conditions, it is possible to prove the global existence of smooth solutions if the initial datum is a small perturbation of a small enough constant state, see \cite{thD,DS12}. In the case of the one dimensional boundary value problem, when the differential part is linearized,  the global existence and the time asymptotic decay of the solutions were proved  in \cite{gumanari}, if the initial data are small perturbations of stable constant stationary  states. To complete the analytical study of the quasilinear model there are some clear difficulties. The first one lies  in  the appearance of regions of vacuum during the evolution of the time solution, since  the hyperbolic part of the model degenerates as the eigenvalues coincide; as far as we know, the only related results are given in \cite{masmoudi, xumin}, and they are about the local existence of solutions for the Euler equations with damping and vacuum, but without chemotaxis.

From a more biological point of view, the appearance of non constant solutions with a succession of regions with high density of cells and regions of vacuum, can be put in correspondence with the formation of patterns,  such as a network of blood capillaries. 
In \cite{NRT}, present authors analyzed the existence of some non-constant steady states to model \eqref{eq:main_system} on a one dimensional bounded domain. In particular, for the pressure law \eqref{eq:pressure_law}  with $\gamma=2$, a complete description of the stationary solutions formed of one region of positive density near the boundary and one region of vacuum
was given.  Numerical simulations also shown that such solutions are stable and can be found as asymptotic states of the system \eqref{eq:main_system} even for strictly positive initial data. In the following, we will call "bump"  a region with a nonnegative density  surrounded by  two regions of vacuum, as shown in blue in Figure \ref{fig:SolStat}, and a "lateral half bump" will be a bump cut in its middle and stuck to an extremity of the interval, as shown in red in the same Figure  \ref{fig:SolStat}.
Other stationary configurations with several bumps have been also observed numerically as asymptotic states of the model in \cite{NRT}. These configurations are described in details with a comparison of their energy values in \cite{BCR}. 

Remark that in the case of bounded domains with no-flux boundary conditions, stationary solutions for both systems \eqref{eq:main_system} and \eqref{eq:KellerSegel} coincide and it is worth exploring if the asymptotic states of the two systems are the same or not. Actually, one may expect the Keller-Segel type model \eqref{eq:KellerSegel}  with $\delta=0$ (i.e.: the parabolic-elliptic case) to be the  large time and large damping limit of the hyperbolic system  \eqref{eq:main_system}, and this is actually the case in \cite{Marcati} for the case    without chemotaxis or  in \cite{DiFrancesco_Donatelli} for our case, both results being proved only on unbounded domains.  In this paper, our main goal is  to make a careful comparison  of the two models    \eqref{eq:main_system} and  \eqref{eq:KellerSegel}  with $\delta=1$,  by analyzing numerically  their actual asymptotic behavior. In particular, we are able to exhibit some sets of   initial data and some parameters such that the  two systems  converge asymptotically  to two different stationary solutions, namely two solutions with a different number of bumps. In that case, the diffusive Keller-Segel model \eqref{eq:KellerSegel} seems  to be more inclined to merge  bumps together, so that the asymptotic solution contains a smaller number of bumps than the asymptotic solution for the hyperbolic system  \eqref{eq:main_system}, often after a long transient where nothing happens. Some similar phenomena  which are referred as metastability of patterns,  were observed in the case of  a  Keller-Segel type model with linear diffusion and a  logistic chemosensitive function in  \cite{HP2001, PH2005, Dolak_Schmeiser}, see Subsection \ref{generic} for more details.

In order to perform such a comparison,  we need  first  to find an accurate scheme for the hyperbolic system to make a reliable comparison of the two models. The approximation of this system needs special care due to the presence of vacuum states and emergence of non-constant steady states. More precisely, the discretization procedure  has to generate non-negative solutions with finite speed of propagation and should  resolve properly the non constant equilibria, characterized by a vanishing flux. Such problems are well-known when dealing with hyperbolic equations with sources, see for instance \cite{Natalini_Ribot, Gosse_chemo2, gosse_chemo3, gosse_book}.  For that purpose,  we consider the well-balanced scheme proposed in \cite{NRT} and based on the Upwinding Sources at Interfaces methodology \cite{BPV,Bouchut_book,Perthame_Simeoni}. In   \cite{NRT},  we used a  hydrostatic reconstruction, introduced by Audusse et al. \cite{Audusse} in the case of shallow-water  equations and by Bouchut, Ounaissa and Perthame \cite{BOP} in the case of Euler equations with large damping. To use this approach, we compute the  reconstructed  interface variables by integrating the equation for  stationary solutions with a constant velocity. According to the form of the equation we integrate, two different reconstructions can be found; both lead to schemes that are consistent with the hyperbolic problem, preserve the  non-negativity of the density, and are exact on non-constant steady states. However, we show in this paper  that only one  of these two schemes  is asymptotically consistent with a conservative scheme for the parabolic model in  the large time and  large damping limit, and therefore it will be the one used in our comparison. 
Another improvement of the scheme described here, with respect with \cite{NRT}, is the implicit treatment of the damping term, which solves more accurately  the vacuum states and the flux on the non constant equilibria.

This paper is organized as follows: after a brief recall  in Section \ref{recall} about the structure of the stationary solutions with vacuum, found in  \cite{NRT} and \cite{BCR},  we propose  in Section \ref{scheme} two different numerical schemes for the quasilinear hyperbolic  system \eqref{eq:main_system} based on well-balanced techniques, with a particular care for their asymptotic preserving property. Then,  in Section \ref{NumHyp}, we show some numerical evidences of the behavior of these schemes   in order to choose a well adapted scheme. Finally, in Section \ref{NumComp}, we present an accurate scheme for the parabolic system \eqref{eq:KellerSegel}, based on the diffusive relaxation techniques of \cite{ANT}, and we perform a careful numerical comparison between the asymptotic solutions for  systems  \eqref{eq:main_system} and for system \eqref{eq:KellerSegel}.

\section{Stationary solutions  with vacuum}\label{recall}

In \cite{NRT} and  \cite{BCR}, we noticed that in the particular case $\gamma=2$, it is possible   to compute explicitly  and  classify the stationary solutions with vacuum of the two systems \eqref{eq:main_system} and \eqref{eq:KellerSegel}, which obviously coincide. Let us recall briefly these results to make the paper almost self-contained. 

Consider system \eqref{eq:main_system} on a one dimensional bounded domain  $[0,L]$  with no-flux  boundary conditions, that is
\begin{equation}\label{boundary_conditions}
\rho_{x}(0,\cdot)=\rho_{x}(L,\cdot)=0,\quad \rho u(0,\cdot)= \rho u(L,\cdot)=0,\quad \phi_{x}(0,\cdot)=\phi_{x}(L,\cdot)=0.
\end{equation}
Notice  that,  under these conditions, the stationary solutions of system  \eqref{eq:main_system} and system \eqref{eq:KellerSegel} coincide. Remark also that, when considering  the evolution problem  (\ref{eq:main_system})  or  \eqref{eq:KellerSegel} with the previous boundary conditions \eqref{boundary_conditions},  the mass of the density is constant in time, namely 
\begin{equation}\label{def_mass}
M=\int_{[0,L]} \rho(x,0) \, dx=\int_{[0,L]} \rho(x,t) \, dx, \, \textrm{ for all } t \geq 0.
\end{equation}
Therefore, the mass $M$ will be considered in what follows as a parameter which characterizes  stationary solutions.

\subsection{Constant solution} The first type of solutions is given by the constant solutions, that is to say, for all domain length $L>0$  and all mass $M>0$, there is a solution defined by $\ds (\rho,u,\phi)=(\frac M L, 0, \frac{aM}{bL}) $, which  is the only constant solution in the space of stationary states. This kind of solution is displayed in green in  Figure  \ref{fig:SolStat}.

\subsection{One lateral half bump} \label{sec:lateral_bump} Let us denote by $\ds\omega=\frac{1}{D}\left(\frac{a\chi}{2\kappa }-b\right)$. We assume that $\omega>0$ and $\ds L>\frac{\pi}{\sqrt{\omega}}$. Then there exists a unique, positive solution  (up to symmetry)  of mass $M$ with only one region of positive density and one region of vacuum,   given by the following expression~:
\begin{subequations}\label{lateral_bump}
\begin{equation}\label{bump:rho}
\rho(x)=\left\{\begin{array}{ll}
\ds \frac{\chi}{2\kappa}\phi(x)+K,& \textrm{ for } x\in[0,\bar{x}],
\medskip\\
0, &\textrm{ for } x\in(\bar{x},L],
\end{array}\right.
\end{equation}
and
\begin{equation}\label{bump:phi}
\phi(x)=\left\{\begin{array}{ll}
\ds \frac{2\kappa b K}{\omega\chi D}\frac{\cos(\sqrt{\omega}x)}{\cos(\sqrt{\omega}\bar{x})}-\frac{aK}{\omega D},& \textrm{ for } x\in[0,\bar{x}],
\medskip\\
\ds -\frac{2\kappa  K}{\chi}
\frac{\cosh(\sqrt{\frac{b}{D}}(x-L))}{\cosh(\sqrt{\frac{b}{D}}(\bar{x}-L))}, 
&\textrm{ for } x\in(\bar{x},L].
\end{array}\right.
\end{equation}

The free boundary point $\bar{x}$ is given by 
the only value $\ds \bar{x}\in\frac{1}{\sqrt{\omega}}(\pi/2,\pi)$ which solves the equation
\begin{equation}\label{bump:x}
\sqrt{\frac{b}{\omega D}}\tan(\sqrt{\omega}\bar{x})=\tanh(\sqrt{\frac{b}{D}}(\bar{x}-L)),
\end{equation}
and the constant $K$ is  equal to
\begin{equation}\label{bump:K}
K=\frac{D}{b}\frac{ M \omega^{3/2}}{\tan(\sqrt{\omega}\bar{x})-\sqrt{\omega}\bar{x}}.
\end{equation}
\end{subequations}
If $\omega<0$, or $\omega>0$, but $\ds L<\frac{\pi}{\sqrt{\omega}}$, then there is no half bump solution to the problem. 

An example of such solution is plotted in red in  Figure  \ref{fig:SolStat}.

\subsection{Other configurations}
In  \cite{NRT}, we also proved that, in the case when $\ds L>\frac{2\pi}{\sqrt{\omega}}$, there exists only one solution composed by a central bump surrounded by two regions of vacuum. This solution is symmetric with respect to the middle point of the interval $[0,L]$ and can be computed by sticking two lateral bumps, calculated on a domain $[0, L/2]$ with a mass $M/2$.  This can be seen  in blue in  Figure  \ref{fig:SolStat}.

In  \cite{BCR}, it is  shown that, when a two bumps solution,  which is made by  a concatenation of two bumps (in the case  $\ds L>\frac{4\pi}{\sqrt{\omega}}$), is computed, then one can find an infinite number of such solutions. Namely, there is one parameter free in the family of such solutions. Moreover, for each solution, an energy can be calculated and it can be proved that the symmetric two bumps solution has the higher energy (in cyan  in Figure  \ref{fig:SolStat}), whereas the solution with no vacuum  for one of the two bumps has the lower one (in black  in Figure  \ref{fig:SolStat}). It is also possible to construct some 3-bumps solutions with 2 parameters, and so on. We expect all these stationary solutions to exist for any other values of $\gamma>1$, but for $\gamma \neq 2$, it is difficult to work, since we have no  more explicit expressions of the solutions. To conclude this section, let us say that,  among all these configurations, the constant solution has the highest energy, whereas  the lateral half bump is the one with lowest energy, see again \cite{BCR}.

In Figure \ref{fig:SolStat}, we can see five different types of stationary solutions for the same parameters of the system, always with $\gamma=2$, same length of the domain $L$ and same mass $M$. The densities are obtained analytically using the above formulas, while the locations of interfaces with vacuum $\bar{x}$ are found by solving equation \eqref{bump:x} numerically.
\begin{figure}
\includegraphics[scale=0.20]{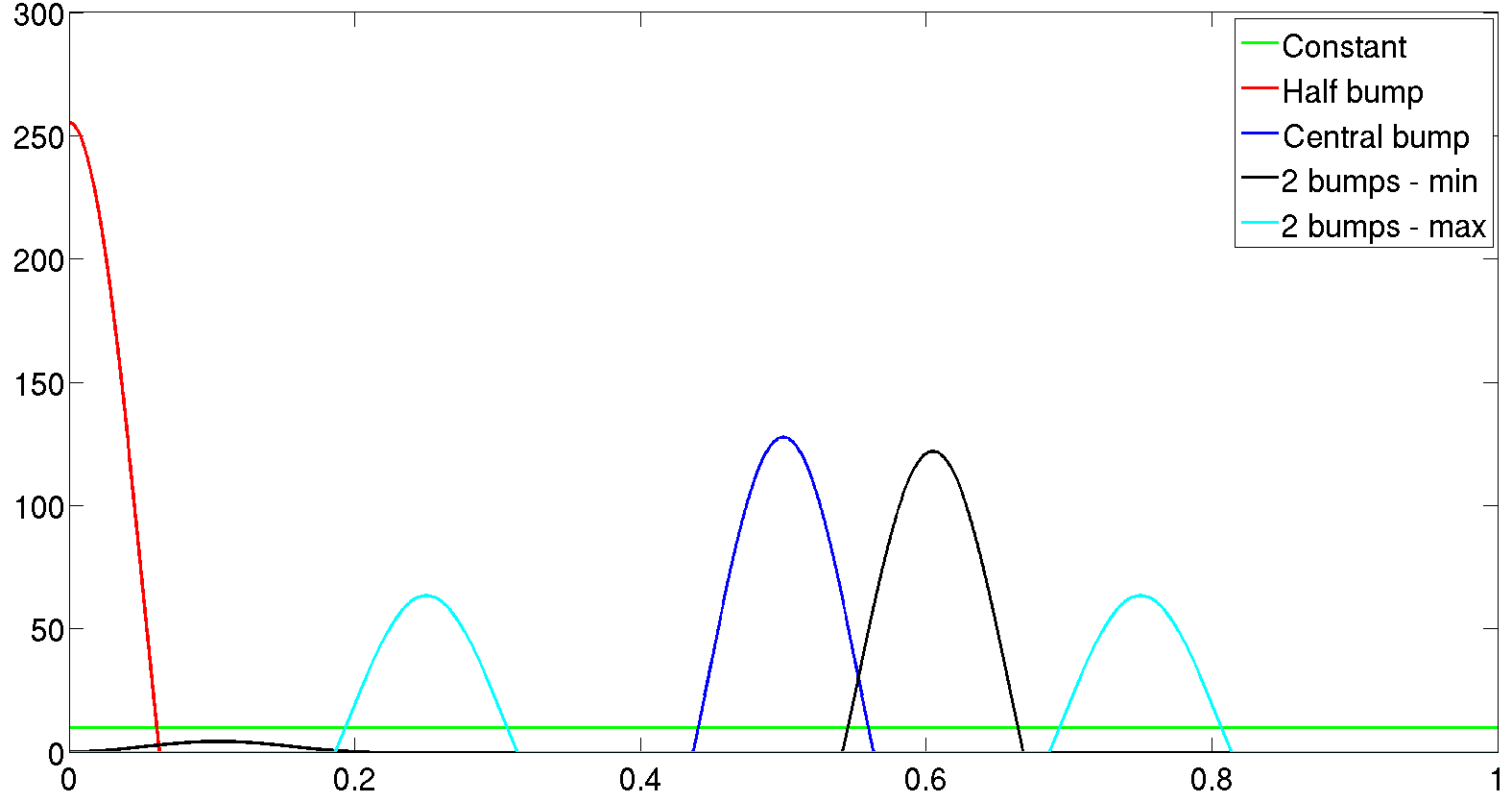}
\caption{Various stationary solutions: constant (in green), one lateral half bump (in red), one central bump (in blue), two bumps with minimal energy (in black), two symmetric bumps with maximal energy (in cyan). The parameters are the following : $\kappa=1,\, \chi=10,\, D=0.1,\, a=20,\, b=10, \, L=1$ and $M=10$.}
\label{fig:SolStat}
\end{figure}

\section{A numerical approximation for the quasilinear hyperbolic model}\label{scheme}

Let us now explain how to construct a well adapted  scheme for system \eqref{eq:main_system}, taking care of being accurate in the approximation of stationary solutions. 
System (\ref{eq:main_system}) couples equations of different natures, i.e. a quasilinear system of conservation laws with sources, which is coupled   with a linear parabolic equation.
The latter is very classical and several methods can be used, as, for instance, finite differences in space and  the classical explicit-implicit Crank-Nicholson method for the time discretization. Now, in this section,  we concentrate on the discretization of the hyperbolic part and we improve the finite volume scheme proposed in \cite{NRT}, based on the Upwinding Sources at the Interfaces technique developed in \cite{Bouchut_book, BPV, Audusse, BOP}.

After a brief recall about  the numerical framework, we analyze the asymptotic preserving property in the large time and large damping  limit, which is given by the parabolic system  \eqref{eq:KellerSegel} with $\delta=0$. We show that E-reconstruction \eqref{eq:Ereconstruction}   is not asymptotic preserving, whereas the P-reconstruction \eqref{eq:Preconstruction}  is consistent, in the large time and large damping  limit, with a conservative scheme for the parabolic problem.  

\subsection{Finite volume, well-balanced scheme}

Denoting by  $U=(\rho,\rho u)^T$ the vector of two unknowns, respectively density and momentum,  the   hyperbolic part of system  \eqref{eq:main_system} can be written in the following form
\begin{subequations}\label{eq:hyp_matrix}
\begin{equation}
\label{eq:hyp_matrix_system} U_{t}+F(U)_{x}=S(U),
\end{equation}
where $F$ is the flux function and $S$ the source term, i.e.
\begin{equation}
\label{eq:hyp_matrix_vectors}F(U)=
\left(
\begin{array}{l}
F^{\rho} \\
F^{\rho u}
\end{array}
\right)
=\left(
\begin{array}{l}\rho u
\\
 \rho u^2+P(\rho)
 \end{array}
\right)
 ,\quad S(U)=
 \left(
\begin{array}{l}
 0
\\ 
-\alpha\rho u+\chi\rho\phi_x
\end{array}
\right).
\end{equation}
\end{subequations}

According to the framework of finite volume schemes, we divide the interval $[0,L]$ into $N$ cells  $C_{i}=[x_{i-1/2},x_{i+1/2})$,  centered at the nodes $x_{i}, \, 1\leq i\leq N$. In the following, we  assume that all the cells have the same length $\ds \Delta x=x_{i+1/2}-x_{i-1/2}$.
We consider  a semi-discrete space approximation of the solution $U$ to system \eqref{eq:hyp_matrix} on cell $C_{i}$, i.e. we compute  an approximation of the cell average of the solution at time $t>0$, that is to  say  
\begin{equation*}
U_{i}(t)=\frac{1}{\Delta x}\int_{x_{i-1/2}}^{x_{i+1/2}} U(x, t) \, dx.
\end{equation*}
We denote now by $ U_{i}^{n}$ an approximation of function $U_{i}(t)$ at discrete time $\ds t_{n}=n \Delta t$, where $\Delta t$ is the time step. 
Let us consider a numerical flux $\mathcal{F}=
\left(
\begin{array}{l}
\mathcal{F}^{\rho} \\
\mathcal{F}^{\rho u}
\end{array}
\right)$   consistent with the continuous flux $F$, that is to say that it satisfies
\begin{equation}\label{consistency}
\mathcal{F}(U,U)=F(U) \; \textrm{ for all } U.
\end{equation}
A general, fully discrete, explicit finite volume scheme for (\ref{eq:hyp_matrix}) can be written as 
\begin{subequations}\label{eq:scheme_main}
\begin{equation}
\frac{\Delta x}{\Delta t}\left(U_{i}^{n+1}-U_{i}^{n}\right)+\mathcal{F}(U_{i+1/2}^{n,-},U_{i+1/2}^{n,+})-\mathcal{F}(U_{i-1/2}^{n,-},U_{i-1/2}^{n,+})=\mathcal{S}_{i}^{n},
\end{equation}
where  the values $U_{i+1/2}^{n,\pm}$ are interface variables computed at the point $x_{i+1/2}$, and $\mathcal{S}_{i}^n$ is an approximation of the source term.

According to  the results in \cite{NRT},  we decided to discretize  the homogeneous part using the Suliciu scheme adapted to vacuum, see also \cite{Bouchut_book}.
Notice that in the following, we may write $\ds\mathcal{F}(\rho_{i+1/2}^{n,-},u_{i+1/2}^{n,-},\rho_{i+1/2}^{n,+},u_{i+1/2}^{n,+}) $ instead of $\ds \mathcal{F}(U_{i+1/2}^{n,-},U_{i+1/2}^{n,+})$ when necessary. 

\subsection{Upwinding Sources at the Interfaces and reconstructions}\label{Reconstruct}

Following  \cite{NRT}, we compute the interfaces variables in order to preserve non-constant stationary solutions with constant velocity. We also need to upwind the source term and we use  the following ansatz~:
\begin{equation}
\mathcal{S}_{i} ^n= \mathcal{S}_{i+1/2}^{n,-}+\mathcal{S}_{i-1/2}^{n,+}= 
\left(\begin{array}{c}
0 \\
P\left(\rho_{i+1/2}^{n,-}\right)-P(\rho_{i}^n)
+P(\rho_{i}^n)-P\left(\rho_{i-1/2}^{n,+}\right)
\end{array}\right).
\end{equation}
\end{subequations}
This ansatz is based on the equation satisfied by  stationary solutions i.e. $F(U)_{x}=S(U)$ and the fact that the momentum of a stationary solution is vanishing, thanks to the boundary conditions \eqref{boundary_conditions}.

We also reconstruct  the interface variables $U_{i+1/2}^{n,\pm}$  using the local equilibrium 
\begin{equation}\label{eq:equilibriumEq}
u_{x}=0,\qquad \left(P(\rho)\right)_{x}=\chi\rho\phi_{x}-\alpha\rho u.
\end{equation}
To derive our first method, we observe that the second equation can be rewritten in terms of the internal energy $e(\rho)$  defined by $\ds e'(\rho)=\frac{P(\rho)}{\rho^2}$, which yields
\begin{equation}\label{defPsi}
(\Psi(\rho)-\chi\phi)_{x}=-\alpha u, \textrm{ with } \Psi(\rho)=e(\rho)+\frac{P(\rho)}{\rho}=\frac{\kappa\gamma}{\gamma-1}\rho^{\gamma-1}.
\end{equation}
From this equation, we obtain  the so-called  E-reconstruction, by an approximate integration on $[x_{i},x_{i+1/2}^{-})$ and $(x_{i+1/2}^{+},x_{i+1}]$, by taking:
\begin{subequations}\label{eq:Ereconstruction}
\begin{equation}
u_{i+1/2}^{n,-}=u_{i}^n,\qquad u_{i+1/2}^{n,+}=u^n_{i+1}
\end{equation}
and
\begin{equation}
\left\{\begin{array}{l}
\Psi\left(\rho_{i+1/2}^{n,-}\right)=\left[\Psi(\rho_{i}^n) + \chi(\min(\phi_{i}^n,\phi_{i+1}^n)-\phi_{i}^n)-\alpha(u_{i}^n)_{+}\Delta x\right]_{+},
\medskip\\
\Psi\left(\rho_{i+1/2}^{n,+}\right)=\left[\Psi(\rho_{i+1}^n) + \chi(\min(\phi_{i}^n,\phi_{i+1}^n)-\phi_{i+1}^n)+\alpha(u_{i+1}^n)_{-}\Delta x\right]_{+}.
\end{array}\right.
\end{equation}
\end{subequations}
We proved in \cite{NRT} that the scheme \eqref{eq:scheme_main}, coupled with the E- reconstruction \eqref{eq:Ereconstruction}, is consistent with system  (\ref{eq:main_system}) away from vacuum, preserves the non-negativity of the density and preserves  the stationary solutions with vanishing velocity.

Another way to proceed  is to integrate directly the equilibrium equation \eqref{eq:equilibriumEq} to get a different reconstruction, namely:
\begin{subequations}\label{eq:Preconstruction}
\begin{equation}
u_{i+1/2}^{n,-}=u_{i}^n,\qquad u_{i+1/2}^{n,+}=u_{i+1}^n
\end{equation}
and
\begin{equation}
\left\{\begin{array}{l}
P\left(\rho_{i+1/2}^{n,-}\right)=\left[P(\rho_{i}^n) + \chi\overline{\rho}_{i+1/2}^n(\min(\phi_{i}^n,\phi_{i+1}^n)-\phi_{i}^n)-\alpha\rho_{i}^n(u_{i}^n)_{+}\Delta x\right]_{+},
\medskip \\
P\left(\rho_{i+1/2}^{n,+}\right)=\Bigl[P(\rho_{i+1}^n) + \chi\overline{\rho}_{i+1/2}^n(\min(\phi_{i}^n,\phi_{i+1}^n)-\phi_{i+1}^n)\\
\qquad \qquad \qquad \qquad \qquad \qquad \qquad \qquad \qquad  \qquad \qquad +\alpha\rho_{i+1}^n(u_{i+1}^n)_{-}\Delta x\Bigr]_{+},
\end{array}\right.
\end{equation}
\end{subequations}
where $\ds \overline{\rho}_{i+1/2}^n=\frac{1}{2}(\rho_{i}^n+\rho_{i+1}^n)$. In the following this scheme will be mentioned as the P-reconstruction.

 The computations  to prove that the  scheme  \eqref{eq:scheme_main}, coupled with  the P-reconstruction \eqref{eq:Preconstruction}  is consistent, preserves the non-negativity of the density and preserves  the stationary solutions with vanishing velocity,  are analogous to the ones performed in \cite{NRT} for the E-reconstruction. Moreover, as shown in \cite{NRT}, this scheme is stable under a condition of the form 
 \begin{equation*}
\sigma(U_{i+1/2}^{n,-},U_{i+1/2}^{n,+})\Delta t\leq\Delta x,
\end{equation*}
where $\ds \sigma(U_{i}, U_{i+1})$ is a numerical speed such that the 
 solver $\ds\mathcal{F}=(\mathcal{F}^{\rho},\mathcal{F}^{\rho u })^t$ for the  homogeneous system $\ds U_{t}+F(U)_{x}=0 $  preserves the non negativity of $\rho$ by interface; see \cite{NRT} or  \cite{Bouchut_book} for more details. 

Next, we propose another possible improvement for our scheme, which consists in discretizing implicitly the damping term.
Indeed, in  the previous  reconstructions \eqref{eq:Ereconstruction} and \eqref{eq:Preconstruction}, the damping term was included in the definition of the  interface densities. However, since it is a stiff term, it is more natural to deal with an  implicit treatment. Therefore,  the damping term is considered in the discretization equation and disappears in the  reconstruction formula;  the second equation of the  fully discrete scheme \eqref{eq:scheme_main} is therefore replaced by~:
\begin{subequations}\label{eq:implicit}
\begin{eqnarray}\label{eq:scheme_implicit}
(\rho u)_{i}^{n+1} =(\rho u)_{i}^{n} &-& \left.\frac{\Delta t}{\Delta x}\left(\mathcal{F}^{\rho u }(\hat{U}_{i+1/2}^{n,-},\hat{U}_{i+1/2}^{n,+})-\mathcal{F}^{\rho u }(\hat{U}_{i-1/2}^{n,-}, \hat{U}_{i-1/2}^{n,+})\right)\right.
\\
\nonumber&+&\left.\frac{\Delta t}{\Delta x}\left(P\left(\hat{\rho}_{i+1/2}^{n,-}\right)-P\left(\hat{\rho}_{i-1/2}^{n,+}\right)\right)
-\alpha \Delta t (\rho u)_{i}^{n+1}, \right.
\end{eqnarray}
with 
\begin{equation}\label{eq:velocity_implicit}
\hat{u}_{i+1/2}^{n,-}=u_{i}^n,\qquad \hat{u}_{i+1/2}^{n,+}=u_{i+1}^n
\end{equation}
and the interface values are given respectively by
\begin{equation}\label{eq:Ereconstruction_implicit}
\left\{\begin{array}{l}
\Psi\left(\hat{\rho}_{i+1/2}^{n,-}\right)=\left[\Psi(\rho_{i}^n) + \chi(\min(\phi_{i}^n,\phi_{i+1}^n)-\phi_{i}^n)\right]_{+},
\qquad\\
\Psi\left(\hat{\rho}_{i+1/2}^{n,+}\right)=\left[\Psi(\rho_{i+1}^n) + \chi(\min(\phi_{i}^n,\phi_{i+1}^n)-\phi_{i+1}^n)\right]_{+},
\end{array}\right.
\end{equation}
for the E-reconstruction, and by 
\begin{equation}\label{eq:Preconstruction_implicit}
\left\{\begin{array}{l}
P\left(\hat{\rho}_{i+1/2}^{n,-}\right)=\left[P(\rho_{i}^n) + \chi\overline{\rho}_{i+1/2}^n(\min(\phi_{i}^n,\phi_{i+1}^n)-\phi_{i}^n)\right]_{+},
\qquad\\
P\left(\hat{\rho}_{i+1/2}^{n,+}\right)=\left[P(\rho_{i+1}^n) + \chi\overline{\rho}_{i+1/2}^n(\min(\phi_{i}^n,\phi_{i+1}^n)-\phi_{i+1}^n)\right]_{+},
\end{array}\right.
\end{equation}
\end{subequations}
for the P-reconstruction. Next, let us study the convergence of these various schemes   in  the large time and large damping limit.

\subsection{The asymptotic preserving property}\label{AP}

We are now going to show that, using a large time and large damping   scaling (LTLD  in the following) in system \eqref{eq:main_system}, we obtain a chemotaxis model based on the  porous medium equation. Recall that, in \cite{Marcati}, it was proven  that 
the velocity of the compressible Euler system without chemotaxis satisfies the Darcy law in the large time limit, while in \cite{DiFrancesco_Donatelli} the asymptotic convergence of system \eqref{eq:main_system} to  parabolic systems was analyzed under various scalings. Both papers were dealing with the case of the whole space.

Here, on the contrary, we are in a bounded interval, and so space dilations are not considered. Therefore, let $\ds\varepsilon=\frac{1}{\alpha}>0$ and define the following scaled variables, to perform the LTLD limit:
\begin{equation*}
\tau=\varepsilon t,\quad
v\ve(x,\tau) = \frac{u(x,t)}{\varepsilon},\quad
\rho\ve(x,\tau) = \rho(x,t),\quad
\phi\ve(x,\tau) = \phi(x,t).
\end{equation*}
Hence, system \eqref{eq:main_system} can be rewritten for the new unknowns as~:
\begin{equation}\label{eq:main_system_rescaled}
\left\{\begin{array}{l}
\ds{\rho\ve_{\tau}+(\rho\ve v\ve)_{x}=0}\\
\ds{\varepsilon^2(\rho\ve v\ve)_{\tau}+\left[\varepsilon^2\rho\ve (v\ve)^2+P(\rho\ve)\right]_{x} = \chi\rho\ve\phi\ve_{x}-\rho\ve v\ve}\\
\ds{\varepsilon\phi\ve_{\tau} = D\phi\ve_{xx}+a\rho\ve-b\phi\ve}
\end{array}\right.
\end{equation}
As $\varepsilon\rightarrow 0$, which corresponds to the LTLD limit, we obtain the following parabolic-elliptic system, see \cite{DiFrancesco_Donatelli} for more analytical details, that is system \eqref{eq:KellerSegel}  with $\delta=0$~:
\begin{equation}\label{eq:parabolic_elliptic}
\left\{\begin{array}{l}
\ds\rho_{\tau}=\left[P(\rho)_{xx}-\chi(\rho\phi_{x})_{x}\right]\medskip
\\
\ds0=D\phi_{xx}+a\rho-b\phi.
\end{array}\right.
\end{equation}
 
 In this section,  we study the asymptotic preserving property of the numerical scheme \eqref{eq:scheme_main} with the reconstruction \eqref{eq:Ereconstruction} or \eqref{eq:Preconstruction} for system \eqref{eq:main_system}. Following \cite{BOP}, we compute  the LTLD limit of the  first component of the numerical flux  $\mathcal{F}^{\rho}$ obtained with  the two methods of reconstruction for  system \eqref{eq:main_system_rescaled} and we analyze their  consistency with  system \eqref{eq:parabolic_elliptic}. 
 More precisely, we say that a numerical flux is asymptotic preserving in the LTLD limit, if we have the following local expansion:
\begin{equation*}
\mathcal{F}^{\rho}\left(\rho_{i+1/2}^{\vr,n,-},u_{i}^n,\rho_{i+1/2}^{\vr,n,+},u_{i+1}^n\right)= \mathcal{F}^{n,\textrm{Par}}_{i+1/2}+O(\vr\Delta x)+O(\vr^2),
\end{equation*} 
where $\mathcal{F}^{n,\textrm{Par}}_{i+1/2}$ is a consistent and conservative numerical flux for the parabolic equation~:
\begin{equation*}
\rho_{t}=\left(P(\rho)_{x}-\chi(\rho\phi_{x})\right)_{x}.
\end{equation*}
We show in the following theorem that this property depends on the reconstruction we choose.

To handle correctly the proof of the theorem, we need the following conditions on the numerical flux~:
\begin{defn}\label{thedef}
A consistent numerical flux function $ \mathcal{F}$  is said to be strongly consistent if it  satisfies the two following conditions~:
\begin{itemize}
\item  if $\ds \mathcal{F}^{\rho u}(r,0,R,0)=P(r)$ then $r=R$;
\item    if $\ds \mathcal{F}^{\rho u}(r,0,R,0)=P(R)$ then $r=R$.
\end{itemize}
\end{defn}

This definition is derived from a condition given in \cite{BOP} and is satisfied, in particular, for the  HLL flux, the HLL-Roe flux  and the Suliciu relaxation flux adapted to vacuum (see \cite{Bouchut_book}). In the annex \ref{annex}, the reader can find a discussion about this condition and the computations to prove that the previous fluxes satisfy it.

\begin{thm}\label{thethm}
Assume  that the hyperbolic numerical flux $\mathcal{F}$ is consistent with the continuous flux $F$ defined in equation \eqref{eq:hyp_matrix_vectors}
and strongly consistent in the sense of Definition \ref{thedef}.
Assume also the following asymptotic expansions for  $\rho_{i}^n$ and $u_{i}^n$ when $\varepsilon\rightarrow 0$:
\begin{equation}\label{eq:asExpansion}
\rho_{i}^{\vr,n}=r_{i}^n+\vr r^{n,(1)}_{i}+O(\vr^2),\qquad u_{i}^{\vr,n}=\vr v_{i}^{\vr,n}=\vr v_{i}^{n,(0)}+\vr^2 v_{i}^{n,(1)}+O(\vr^3).
\end{equation}
Then, we have that
\begin{equation*}
\begin{split}
\mathcal{F}^{\rho}\left(\rho_{i+1/2}^{\vr,-},u_{i},\rho_{i+1/2}^{\vr,+},u_{i+1}\right)&=-\frac{1}{\Delta x }\left( P(r_{i+1}) -P(r_{i}) + \chi\overline{r}_{i+1/2}(\phi_{i}-\phi_{i+1})\right) \\
&+O(\vr\Delta x)+O(\vr^2)
\end{split}
\end{equation*}
in the case of the P-reconstruction and 
\begin{equation*}
\begin{split}
\mathcal{F}^{\rho}\left(\rho_{i+1/2}^{\vr,-},u_{i},\rho_{i+1/2}^{\vr,+},u_{i+1}\right)&=-\frac{r_{i}}{\Delta x }\left( \Psi(r_{i+1}) -\Psi(r_{i}) + \chi(\phi_{i}-\phi_{i+1})\right)\\
&+O(\vr\Delta x)+O(\vr^2).
\end{split}
\end{equation*}
in the case of the E-reconstruction. 

Therefore, in the case of the isentropic gases pressure \eqref{eq:pressure_law}, only the P-reconstruction is asymptotic preserving in the LTLD limit.
\end{thm}

\begin{proof}
We first pass to the limit   $\vr \to 0$  in the second equation of scheme  \eqref{eq:scheme_main}, using expansions  \eqref{eq:asExpansion}, and we obtain that 
\begin{equation*}
\mathcal{F}^{\rho u}(r_{i+1/2}^{n,-},0,r_{i+1/2}^{n,+},0)-\mathcal{F}^{\rho u}(r_{i-1/2}^{n,-},0,r_{i-1/2}^{n,+},0)= P\left(r_{i+1/2}^{n,-}\right)-P\left(r_{i-1/2}^{n,+}\right).
\end{equation*}
Thanks to the assumption of strong consistency of the flux,
 the relation $r_{i+1/2}^{n,-}=r_{i+1/2}^{n,+}$ gives a unique solution to the previous equation; see annex \ref{annex} for more details. 

We now pass into the limit $\vr \to 0$  in the first equation of scheme  \eqref{eq:scheme_main}. 
The asymptotic expansion of $\mathcal{F}^{\rho}\left(\rho_{i+1/2}^{\vr,n,-},u_{i}^n,\rho_{i+1/2}^{\vr,n,+},u_{i+1}^n\right)$ around a state $\left(\rho_{i+1/2}^{\vr,n,-},u_{i}^n,\rho_{i+1/2}^{\vr,n,-},u_{i}^n\right)$ is, dropping the $n$ index when there is no confusion:
\begin{equation}\label{Estimate}
\begin{split}
&\mathcal{F}^{\rho}\left(\rho_{i+1/2}^{\vr,-},u_{i},\rho_{i+1/2}^{\vr,+},u_{i+1}\right)= \mathcal{F}^{\rho}\left(\rho_{i+1/2}^{\vr,-},u_{i},\rho_{i+1/2}^{\vr,-},u_{i}\right)\\
& +(\rho_{i+1/2}^{\vr,+}-\rho_{i+1/2}^{\vr,-})\left(\partial_{3}\mathcal{F}^{\rho}\right)+  (u_{i+1}-u_{i})\left(\partial_{4}\mathcal{F}^{\rho}\right)  +O(\rho_{i+1/2}^{\vr,+}-\rho_{i+1/2}^{\vr,-})^2\\
&+O  (u_{i+1}-u_{i})^2,
\end{split}
\end{equation}
where $\ds \partial_{i}\mathcal{F}^{\rho}$ is the derivative of $\mathcal{F}^{\rho}$ with respect to the $i$-th variable calculated at point $\left(\rho_{i+1/2}^{\vr,n,-},u_{i}^n,\rho_{i+1/2}^{\vr,n,-},u_{i}^n\right)$ .

Let us now estimate all these terms separately. It is first easy to see that $\ds  (u_{i+1}-u_{i})=\vr \Delta x$, from expansions \eqref{eq:asExpansion}. Now, using the consistency of $\ds \mathcal{F}^{\rho}$, we have :
\begin{equation*}
\mathcal{F}^{\rho}\left(\rho_{i+1/2}^{\vr,-},u_{i},\rho_{i+1/2}^{\vr,-},u_{i}\right) = F^{\rho}(\rho_{i+1/2}^{\vr,-},u_{i})=\rho_{i+1/2}^{\vr,-}u_{i}=(\rho_{i+1/2}^{\vr,-}-\rho_{i}^\vr)u_{i}+ \rho_{i}^\vr u_{i}.
\end{equation*}
Using the P-reconstruction \eqref{eq:Preconstruction}  or the E-reconstruction \eqref{eq:Ereconstruction} for $\rho_{i+1/2}^{\vr,\pm}$ and the asymptotic expansion \eqref{eq:asExpansion},  we have $\ds(\rho_{i+1/2}^{\vr,-}-\rho_{i}^\vr)u_{i}=  O(\vr\Delta x)$.

\textit{Case of the P-reconstruction. }
Let us now estimate the two remaining  terms  $\ds \rho_{i}^\vr u_{i}$ and $\ds  \rho_{i+1/2}^{\vr,+}-\rho_{i+1/2}^{\vr,-}$ in the case of the P-reconstruction. 
Passing to the limit in the reconstruction \eqref{eq:Preconstruction}  when $\vr \to 0$, we obtain
\begin{equation*}
\left\{\begin{array}{l}
P\left(r_{i+1/2}^{-}\right)=\left[P(r_{i}) + \chi\overline{r}_{i+1/2}(\min(\phi_{i},\phi_{i+1})-\phi_{i})- r_{i}(u_{i})_{+}\Delta x\right]_{+},\\
P\left(r_{i+1/2}^{+}\right)=\left[P(r_{i+1}) + \chi\overline{r}_{i+1/2}(\min(\phi_{i},\phi_{i+1})-\phi_{i+1})+ r_{i+1} (u_{i+1})_{-}\Delta x\right]_{+}
\end{array}\right.
\end{equation*}
and using that $r_{i+1/2}^{n,-}=r_{i+1/2}^{n,+}$, we obtain by subtracting the two previous equations and dropping the positive part:
\begin{equation}\label{DiffReconstrPLimit}
0=P(r_{i+1}) -P(r_{i}) + \chi\overline{r}_{i+1/2}(\phi_{i}-\phi_{i+1})+ (r_{i+1}(u_{i+1})_{-}+ r_{i}(u_{i})_{+})\Delta x.
\end{equation}
By the same way, subtracting the two equations \eqref{eq:Preconstruction} and dropping the positive part, we also obtain the following  equation~:
\begin{equation}\label{DiffReconstrP}
\begin{split}
P\left(\rho_{i+1/2}^{\vr,+}\right)&-P\left(\rho_{i+1/2}^{\vr,-}\right)=P(\rho_{i+1}^\vr) -P(\rho_{i}^\vr)
+ \chi\overline{\rho}_{i+1/2}^\vr(\phi_{i}-\phi_{i+1})\\
&+(\rho_{i+1}^\vr(u_{i+1})_{-}+\rho_{i}^\vr(u_{i})_{+})\Delta x.
\end{split}
\end{equation}
Now, subtracting \eqref{DiffReconstrP} and \eqref{DiffReconstrPLimit} gives~:
\begin{equation*}
\begin{split}
P\left(\rho_{i+1/2}^{\vr,+}\right)&-P\left(\rho_{i+1/2}^{\vr,-}\right)=(P(\rho_{i+1}^\vr) -P(\rho_{i}^\vr))-(P(r_{i+1}) -P(r_{i}) )\\
&+((\rho_{i+1}^\vr-r_{i+1})(u_{i+1})_{-}+(\rho_{i}^\vr-r_{i})(u_{i})_{+})\Delta x\\
& + \chi(\overline{\rho}_{i+1/2}^\vr-\overline{r}_{i+1/2})(\phi_{i}-\phi_{i+1}).
\end{split}
\end{equation*}
Under this previous form and using the expansions  \eqref{eq:asExpansion}, it is straightforward that $\ds P\left(\rho_{i+1/2}^{\vr,+}\right)-P\left(\rho_{i+1/2}^{\vr,-}\right)=O(\vr\Delta x)+O(\vr^2)$ and, consequently,  that  $\ds \rho_{i+1/2}^{\vr,+}-\rho_{i+1/2}^{\vr,-}=O(\vr\Delta x)+O(\vr^2)$. 

Finally, let us consider the last term $\ds \rho_{i}^\vr u_{i}$ of \eqref{Estimate}.
Using expansions \eqref{eq:asExpansion}, we decompose it as follows, in order to use relation \eqref{DiffReconstrPLimit}~:
\begin{equation*}
\begin{split}
 \rho_{i}^\vr u_{i}&= \rho_{i}^\vr ((u_{i})_{+}+(u_{i})_{-})=\rho_{i}^\vr (u_{i})_{+}+\rho_{i+1}^\vr (u_{i+1})_{-}+ O(\vr \Delta x),\\
 &=r_{i} (u_{i})_{+}+r_{i+1} (u_{i+1})_{-}+O(\vr^2)+ O(\vr \Delta x)
 \end{split}
\end{equation*}
which gives
\begin{equation*}
 \rho_{i}^\vr u_{i}=-\frac{1}{ \Delta x }\left( P(r_{i+1}) -P(r_{i}) + \chi\overline{r}_{i+1/2}(\phi_{i}-\phi_{i+1})\right)
 +O(\vr^2)+ O(\vr \Delta x).
\end{equation*}
To conclude, all these estimates enable us to write equation \eqref{Estimate}  in the case of the P-reconstruction as~:
\begin{equation*}
\begin{split}
\mathcal{F}^{\rho}\left(\rho_{i+1/2}^{\vr,-},u_{i},\rho_{i+1/2}^{\vr,+},u_{i+1}\right)&=-\frac{1}{ \Delta x }\left( P(r_{i+1}) -P(r_{i}) + \chi\overline{r}_{i+1/2}(\phi_{i}-\phi_{i+1})\right) \\
&+O(\vr\Delta x)+O(\vr^2).
\end{split}
\end{equation*}

\textit{Case of the E-reconstruction. } We follow the same computations as for the P-reconstruction.
We estimate the two remaining  terms  $\ds \rho_{i}^\vr u_{i}$ and $\ds  \rho_{i+1/2}^{\vr,+}-\rho_{i+1/2}^{\vr,-}$ in the case of the E-reconstruction. 
Passing to the limit in the reconstruction \eqref{eq:Ereconstruction}  when $\vr \to 0$, we obtain
\begin{equation*}
\left\{\begin{array}{l}
\Psi\left(r_{i+1/2}^{-}\right)=\left[\Psi(r_{i}^n) + \chi(\min(\phi_{i}^n,\phi_{i+1}^n)-\phi_{i}^n)-(u_{i}^n)_{+}\Delta x\right]_{+},\\
\Psi\left(r_{i+1/2}^{+}\right)=\left[\Psi(r_{i+1}^n) + \chi(\min(\phi_{i}^n,\phi_{i+1}^n)-\phi_{i+1}^n)+(u_{i+1}^n)_{-}\Delta x\right]_{+}.
\end{array}\right.
\end{equation*}
and using that $r_{i+1/2}^{n,-}=r_{i+1/2}^{n,+}$, we obtain by subtracting the two previous equations and dropping the positive part:
\begin{equation}\label{DiffReconstrELimit}
0=\Psi(r_{i+1}) -\Psi(r_{i}) + \chi(\phi_{i}-\phi_{i+1})+ ((u_{i+1})_{-}+ (u_{i})_{+})\Delta x.
\end{equation}
By the same way, subtracting the two equations \eqref{eq:Ereconstruction} and dropping the positive part, we also obtain the following  equation~:
\begin{equation}\label{DiffReconstrE}
\begin{split}
\Psi\left(\rho_{i+1/2}^{\vr,+}\right)&-\Psi\left(\rho_{i+1/2}^{\vr,-}\right)=\Psi(\rho_{i+1}^\vr) -\Psi(\rho_{i}^\vr)
+ \chi(\phi_{i}-\phi_{i+1})\\
&+((u_{i+1})_{-}+(u_{i})_{+})\Delta x.
\end{split}
\end{equation}
Now, subtracting \eqref{DiffReconstrE} and \eqref{DiffReconstrELimit} gives~:
\begin{equation*}
\Psi\left(\rho_{i+1/2}^{\vr,+}\right)-\Psi\left(\rho_{i+1/2}^{\vr,-}\right)=(\Psi(\rho_{i+1}^\vr) -\Psi(\rho_{i}^\vr))-(\Psi(r_{i+1}) -\Psi(r_{i}) ).
\end{equation*}
Using the expansions  \eqref{eq:asExpansion}, it is straightforward that $\ds \rho_{i+1/2}^{\vr,+}-\rho_{i+1/2}^{\vr,-}=O(\vr\Delta x)+O(\vr^2)$. 

Finally, let us consider the last term $\ds \rho_{i}^\vr u_{i}$ of \eqref{Estimate}.
Using expansions \eqref{eq:asExpansion}, we decompose it as follows, in order to use relation \eqref{DiffReconstrELimit}~:
\begin{equation*}
\begin{split}
 \rho_{i}^\vr u_{i}&= \rho_{i}^\vr ((u_{i})_{+}+(u_{i})_{-})=\rho_{i}^\vr ((u_{i})_{+}+(u_{i+1})_{-})+ O(\vr \Delta x),\\
 &=r_{i} ((u_{i})_{+}+(u_{i+1})_{-})+O(\vr^2)+ O(\vr \Delta x)
 \end{split}
\end{equation*}
which gives
\begin{equation*}
 \rho_{i}^\vr u_{i}=-\frac{r_{i}}{ \Delta x }\left( \Psi(r_{i+1}) -\Psi(r_{i}) + \chi(\phi_{i}-\phi_{i+1})\right)
 +O(\vr^2)+ O(\vr \Delta x).
\end{equation*}
To conclude, all these estimates enable us to write equation \eqref{Estimate}  in the case of the E-reconstruction as~:
\begin{equation*}
\begin{split}
\mathcal{F}^{\rho}\left(\rho_{i+1/2}^{\vr,-},u_{i},\rho_{i+1/2}^{\vr,+},u_{i+1}\right)&=-\frac{r_{i}}{ \Delta x }\left( \Psi(r_{i+1}) -\Psi(r_{i}) + \chi(\phi_{i}-\phi_{i+1})\right)\\
&+O(\vr\Delta x)+O(\vr^2).
\end{split}
\end{equation*}
\end{proof}
Following \eqref{defPsi}, we can see that the second result of the theorem can be rewritten as~:
\begin{equation*}
\begin{split}
\mathcal{F}^{\rho}\left(\rho_{i+1/2}^{\vr,-},u_{i},\rho_{i+1/2}^{\vr,+},u_{i+1}\right)&=-\frac{r_{i}}{ \Delta x }\left( \frac{\kappa\gamma}{\gamma-1}(r_{i+1}^{\gamma-1}-r_{i}^{\gamma-1}) + \chi(\phi_{i}-\phi_{i+1})\right)\\
&+O(\vr\Delta x)+O(\vr^2).
\end{split}
\end{equation*}
in the case of the E-reconstruction and of the pressure law for isentropic gases \eqref{eq:pressure_law}. This means that,  using the E-reconstruction \eqref{eq:Ereconstruction},  we obtain in the  LTLD limit,  a numerical scheme for the parabolic-elliptic equation \eqref{eq:parabolic_elliptic} which is non-conservative and which has a wrong diffusion coefficient. On the other hand, the P-reconstruction  \eqref{eq:Preconstruction} used in the numerical scheme \eqref{eq:scheme_main} is consistent with the second order conservative scheme for the Keller-Segel type model  \eqref{eq:parabolic_elliptic}. In the following, we would therefore rather use the P-reconstruction in our numerical tests. 

\section{Numerical tests for the hyperbolic model}\label{NumHyp}

In this section we compare numerically the  methods  presented in the previous section. We show that the implicit treatment of the damping term gives a better approximation of the momentum and of the density near the vacuum states. Then, choosing the P-reconstruction and the implicit treatment of the damping, we analyze the dependence of the asymptotic numerical solutions on the mesh refinement .  

Let us observe that, in the following, we will display several plots of the residuals of the density with respect to time in a log-log scale. The residuals at discrete time $t^n=n\Delta t$ are defined by the difference of the density between time $t^{n+1}$ and time $t^{n}$, that is to say $\ds || \rho^{n+1}-\rho^n ||$, and the evolution of the residuals stops whenever we reach a stationary asymptotic solution.

To begin with, let us   compare implicit and explicit treatment of the damping term in the momentum balance equation. In what follows, we will use the Suliciu relaxation scheme adapted to vacuum  for the homogeneous part of the hyperbolic system for $\rho$  in  \eqref{eq:main_system} and the parabolic equation  for $\phi$ of system \eqref{eq:main_system} is treated thanks to a classical centered discretization   in space and a Crank-Nicolson scheme  in time.

\subsection{Implicit vs. explicit approximations of the damping term}
In section \ref{Reconstruct}, we proposed a partially implicit version of the well-balanced scheme, see equation \eqref{eq:implicit},  by treating in a different way the linear stiff damping term of the  momentum equation of system \eqref{eq:main_system}. Indeed  in the well-balanced scheme presented in \cite{NRT}, we treated this term  explicitly inside the new reconstructed variables. In this test,  we study the effect of this implicit treatment. 

We consider  system \eqref{eq:main_system} on an interval  of length $L=1$  with the following parameters :  $\gamma=2$ or $\gamma=3$,  $\chi=50$, $D=a=b=\alpha=\kappa=1$. We take as initial datum  the   density defined as $\rho_{0}(x)=1+\sin{(4\pi|x-0.25L|)}$. Figure~\ref{fig:test1_gamma2} (respectively  Figure~\ref{fig:test1_gamma3}) displays  the  solution to  system  \eqref{eq:main_system}   for $\gamma=2$ (resp. $\gamma=3$) at time $T=300$  obtained with the various well-balanced schemes presented in section  \ref{Reconstruct}. More precisely, the figures on the top show the density $\rho$  and the  concentration $\phi$, whereas the subfigures on the bottom show the logarithm of the momentum, $\log|\rho u |$ as a function of the space variable $x$. On the left, the P-reconstruction \eqref{eq:Preconstruction}  or   \eqref{eq:Preconstruction_implicit} is used, whereas on the right, the reconstruction is given by the  E-reconstruction  \eqref{eq:Ereconstruction}  or   \eqref{eq:Ereconstruction_implicit}. Finally, on each subfigure, the implicit treatment of the damping term, given by scheme \eqref{eq:implicit}   is compared with its explicit integration, that is to say using the schemes \eqref{eq:scheme_main}-\eqref{eq:Preconstruction}-\eqref{eq:Ereconstruction} .
\begin{figure}[htbp!]
\centering
\begin{tabular}{cc}
\includegraphics[scale=0.1]{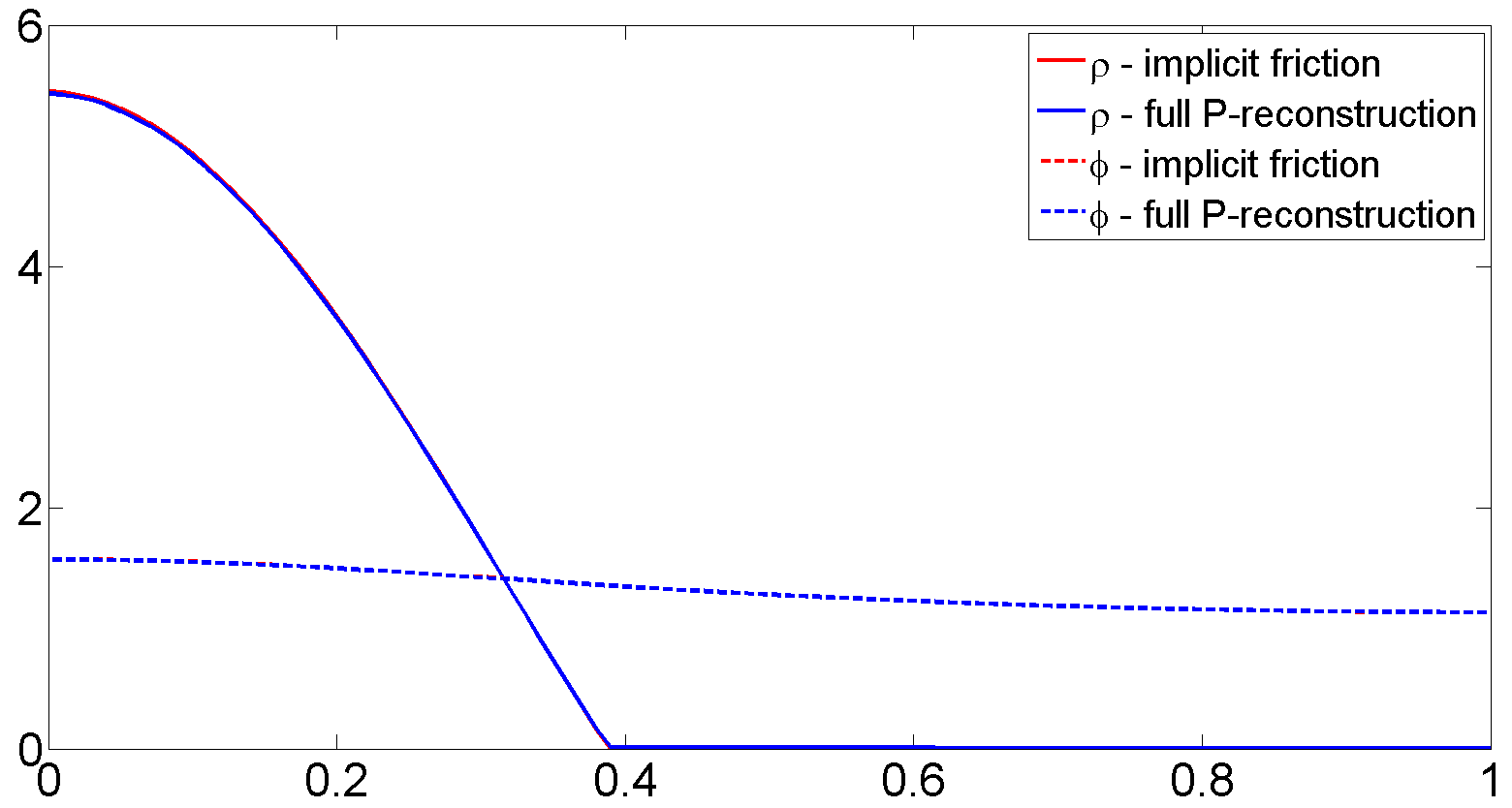}&\includegraphics[scale=0.1]{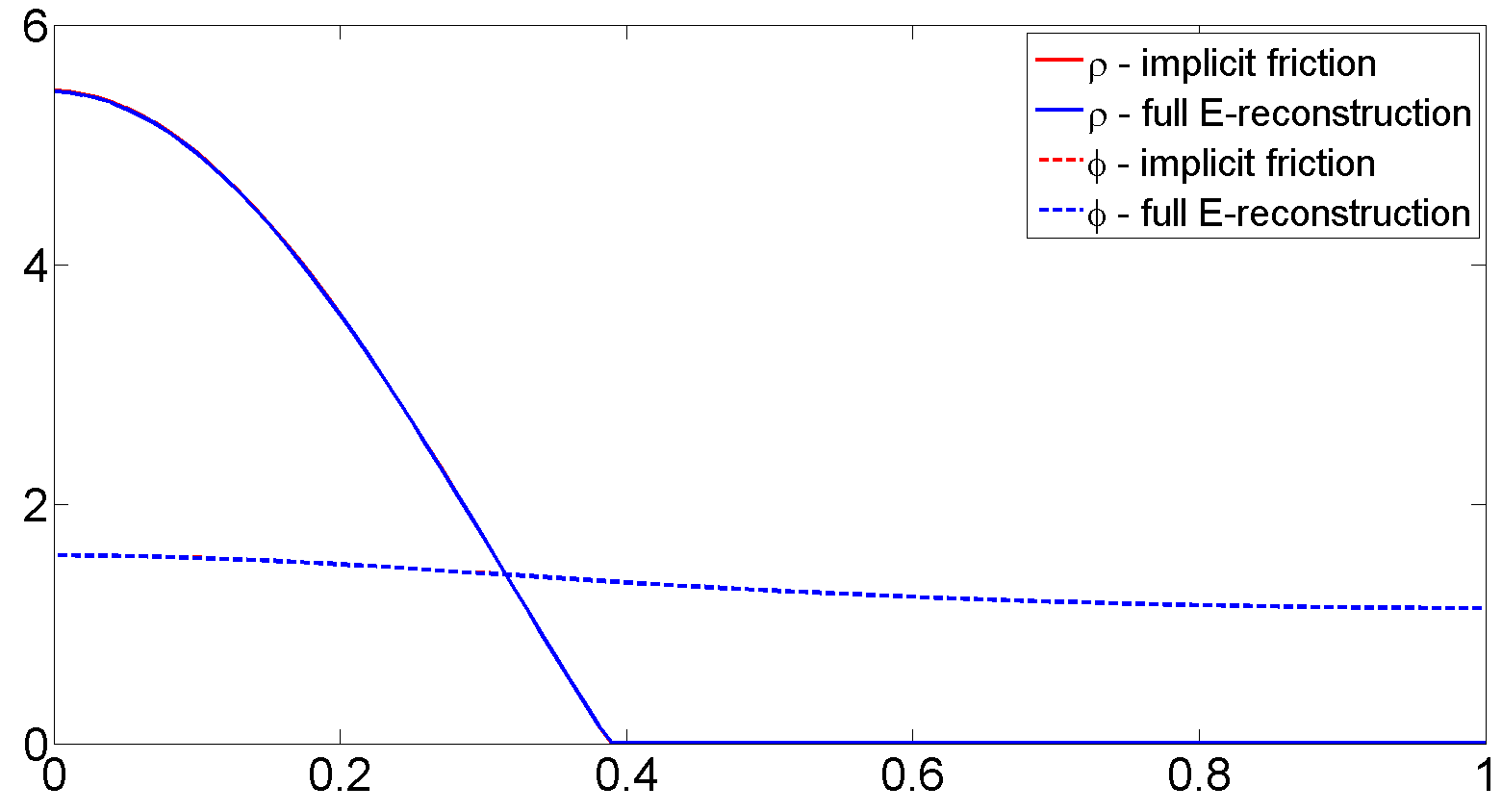}\\
P-reconstruction: $\rho,\phi$ & E-reconstruction: $\rho,\phi$\\
\includegraphics[scale=0.1]{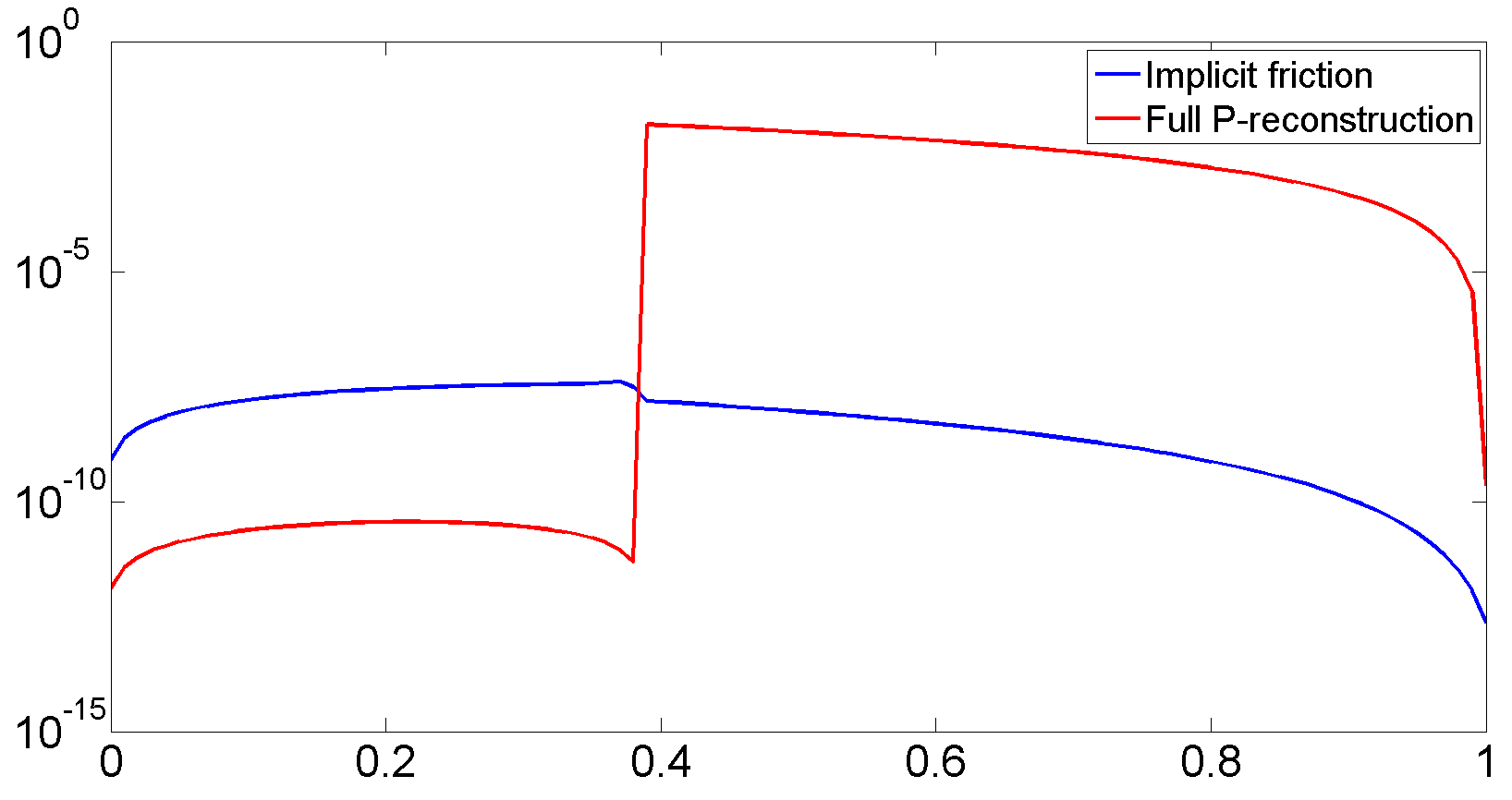}&\includegraphics[scale=0.1]{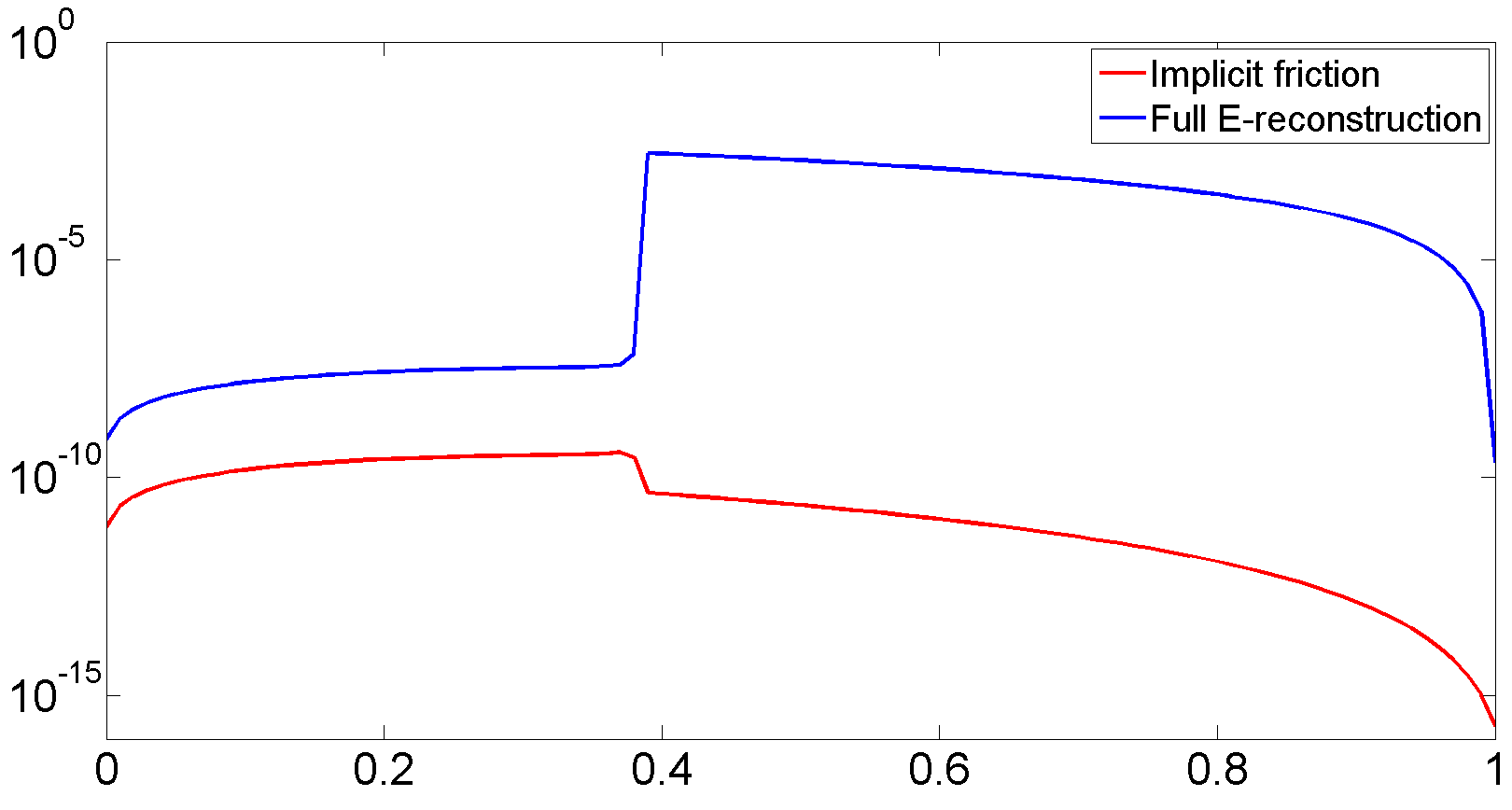}\\
P-reconstruction: $\log|\rho u|$ & E-reconstruction: $\log|\rho u|$
\end{tabular}
\caption{Case of the lateral bump with  $\gamma=2$: Density and concentration distributions (on top) and the logarithm of the momentum (on bottom) for the model \eqref{eq:main_system} with  $\chi=50$, $D=a=b=\alpha=\kappa=1$ and $\gamma=2$. Comparison between the P-reconstruction (on the left) and the E-reconstruction (on the right). On each subfigure,  the implicit treatment of the damping term is compared with the explicit approximation.}
\label{fig:test1_gamma2}
\end{figure}
\begin{figure}[htbp!]
\begin{center}
\begin{tabular}{cc}
\includegraphics[scale=0.1]{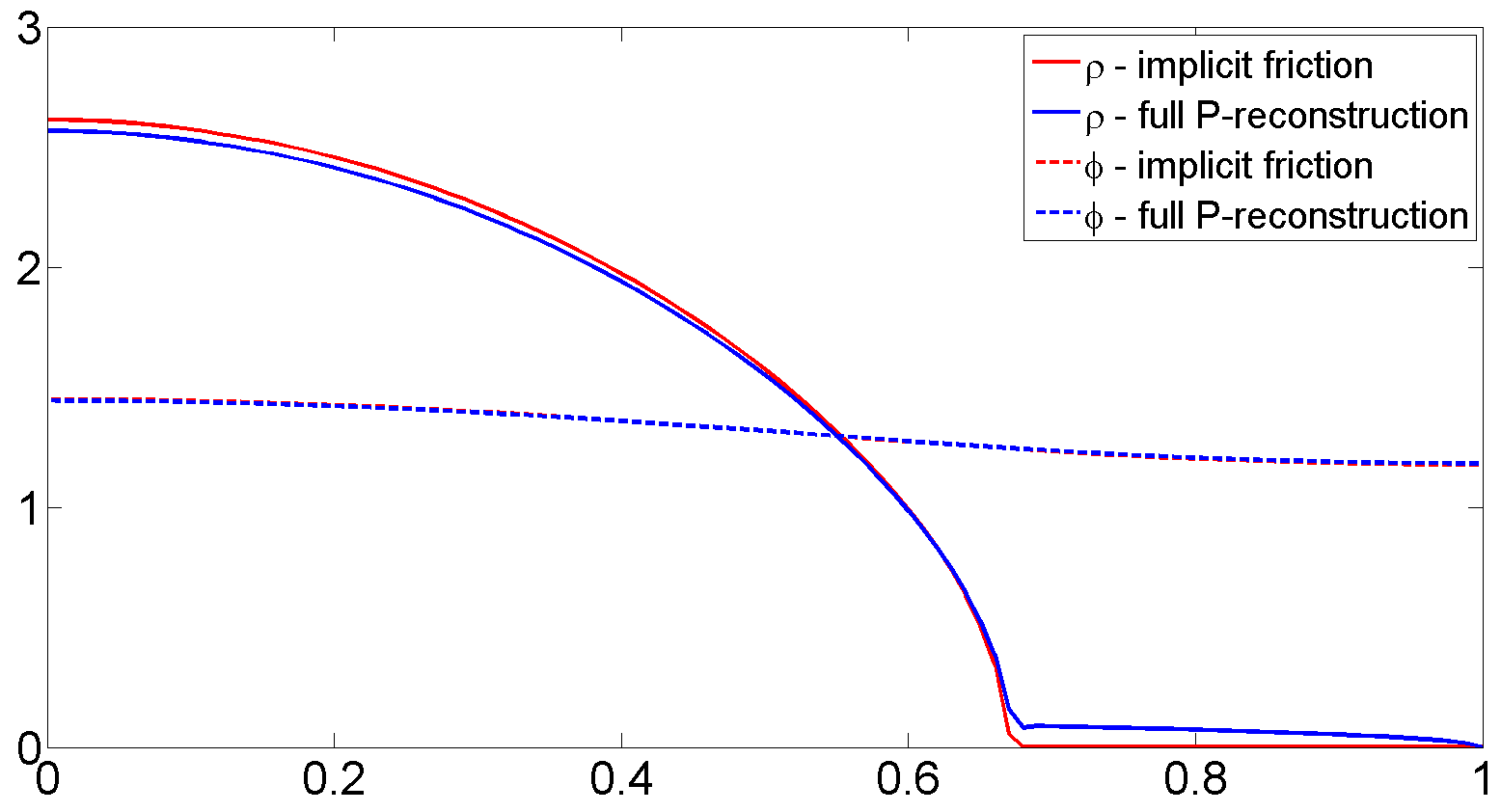}&\includegraphics[scale=0.1]{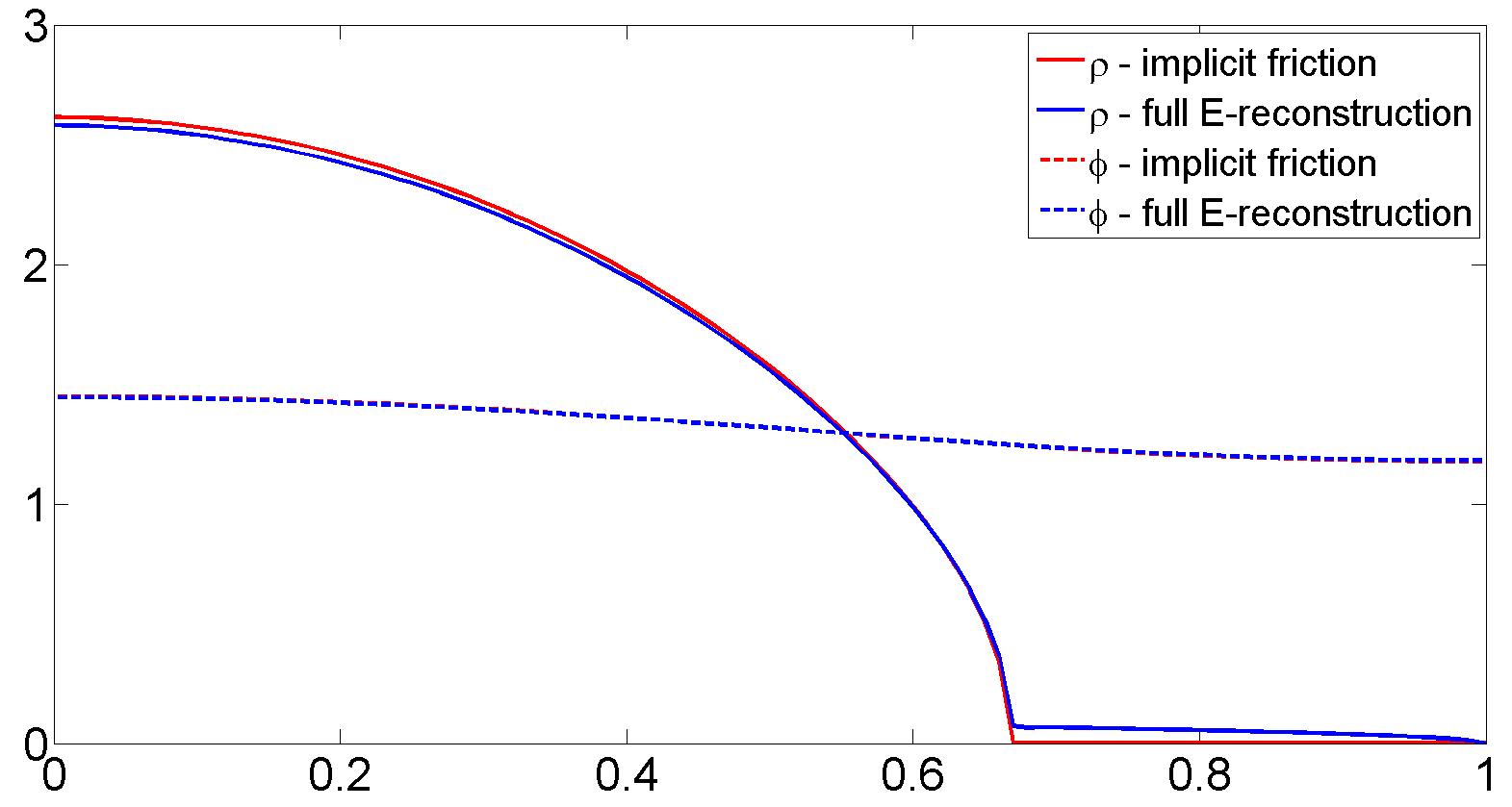}\\
P-reconstruction: $\rho,\phi$ & E-reconstruction: $\rho,\phi$\\
\includegraphics[scale=0.1]{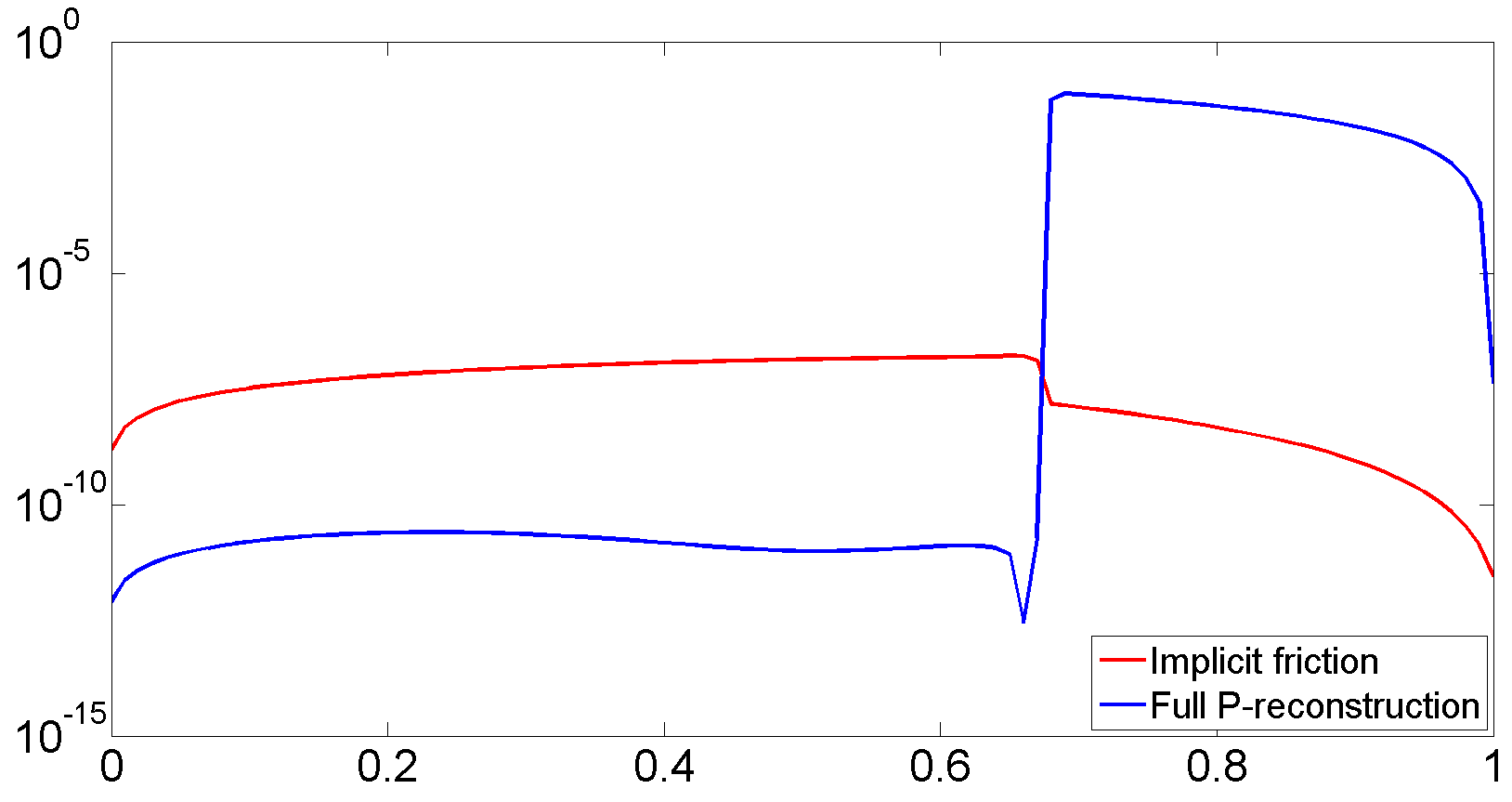}&\includegraphics[scale=0.1]{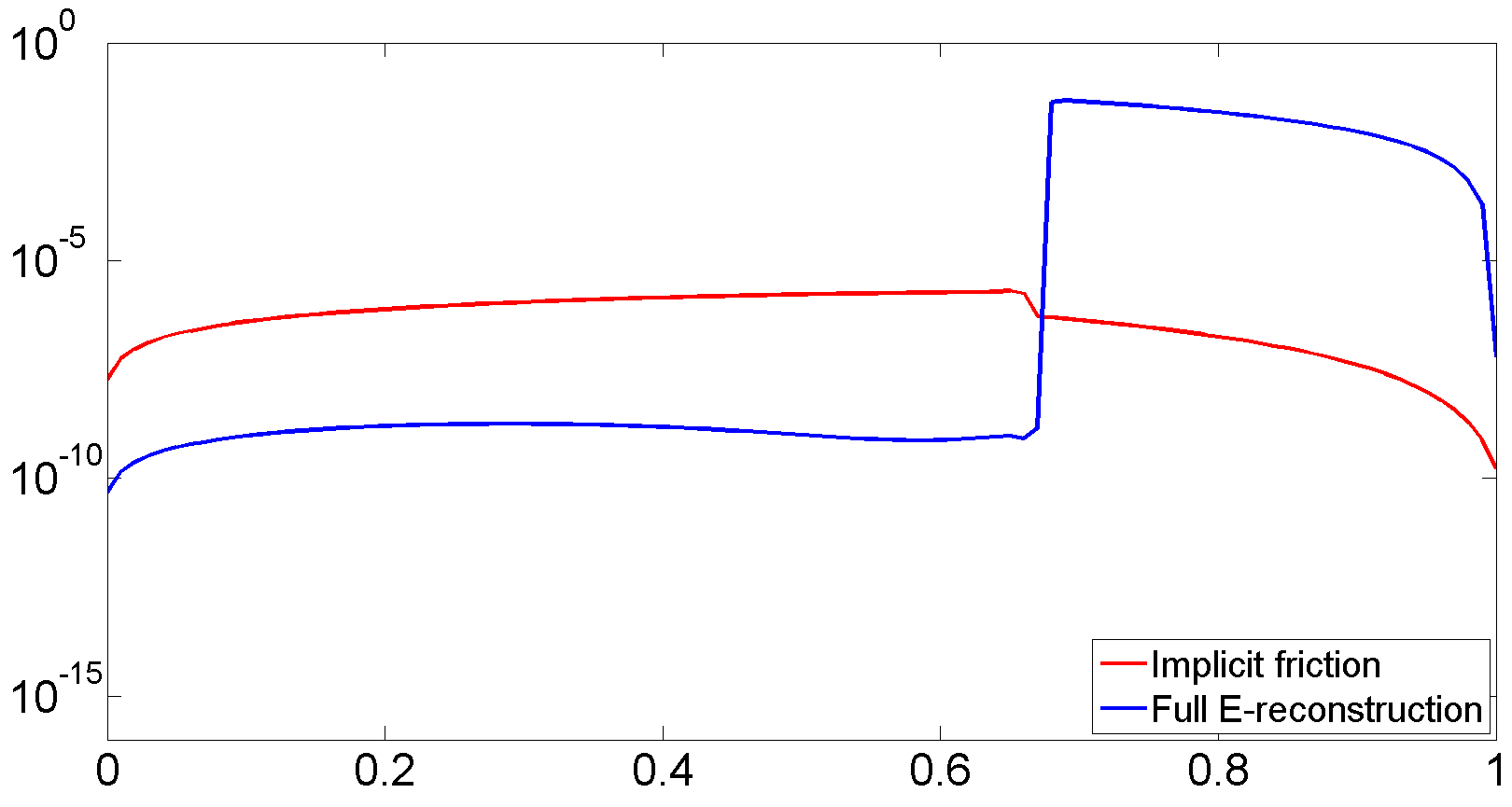}\\
P-reconstruction: $\log|\rho u|$ & E-reconstruction: $\log|\rho u|$
\end{tabular}
\end{center}
\caption{Case of the lateral bump with $\gamma=3$: Density and concentration distributions (on top) and the logarithm  of the momentum (on bottom) for the model \eqref{eq:main_system} with  $\chi=50$, $D=a=b=\alpha=\kappa=1$ and $\gamma=3$. Comparison between the P-reconstruction (on the left) and the  E-reconstruction (on the right). On each subfigures,  the implicit treatment of the damping term is compared with the explicit approximation.}
\label{fig:test1_gamma3}
\end{figure}

We notice that the two reconstruction methods give comparable results; however, there is a large improvement  when the damping term is treated implicitly. It is especially visible for the momentum function, which should vanish asymptotically. In particular, it gives better accuracy  in the region where the density vanishes. For the explicit approximation,  in the case $\gamma=3$,   the density and the momentum are far from being zero at the theoretical vacuum states. We remark that in this case, the implicit approximation reduces the $L^{\infty}$ error of the momentum of an  order $10^6-10^8$. 

In  subsection \ref{AP}, we have  shown that only the P-reconstruction is asymptotically consistent with a numerical scheme for the correct limit parabolic model. This is the reason why, from now on, we will use the finite volume scheme \eqref{eq:scheme_main} with the P-reconstruction \eqref{eq:Preconstruction_implicit} and the implicit treatment of the damping term.

\subsection{Mesh dependence}

In computational fluid dynamics (CFD) it is well known that characteristic phenomena of the flow appear at different length scales, and so we have to study the influence of grid refinement on the numerical approximation of solutions. Unlike standard fluid dynamics problems, for which the areas where small scale phenomena may occur are known and static,  or their evolution is known, so that adaptive mesh refinement can be applied, for model \eqref{eq:main_system} of chemotaxis,  the situation is more complicated since we do not have any information on the precise location of the critical regions, especially  vacuum regions. In order to compare correctly the asymptotic behavior of the hyperbolic and parabolic models of chemotaxis in section \ref{NumComp}, we analyze first  the effect of the mesh refinement  on the solutions to the hyperbolic model using the iterative refining process. At  first step, we run the simulation on a coarse mesh up to the convergence, that is until a steady solution with vanishing velocity is reached. Then, we refine the mesh uniformly and run the simulation again. If the successive result bears sufficient similarities, the iterative refinement is stopped and the first mesh can be considered as an accurate one. Otherwise, we repeat the mesh refinement until the benefit gained by using the finest mesh is no more significant. 


As  mentioned before, we use the implicit scheme  \eqref{eq:implicit} with the P-reconstruction \eqref{eq:Preconstruction_implicit}. It is difficult to choose a good mesh-dependence test that would reflect all the possible behaviors of the model. This is why we analyze the case which will be used as a  comparison between the parabolic and the hyperbolic models.  We consider  system \eqref{eq:main_system} on an interval of length $L=3$  with the following parameters  $\gamma=2$ or $\gamma=3$, $\chi=10$, $D=0.1$, $a=20$, $b=10$, $\alpha=\kappa=1$ . We take  as an initial datum,  the density $\rho_{0}(x)=1.5+\sin{(4\pi|x-0.25L|)}$ and the initial concentration $\phi=0$. Figure~\ref{fig:meshRefinement_gamma2} (resp. \ref{fig:meshRefinement_gamma3}) shows, on the left, the density as a function of space, using different  space steps in the case $\gamma=2$ (resp. $\gamma=3$). On the right, we display  the evolution of the residuals of the density with respect to time using  a log-log scale.  
\begin{figure}[htbp!]
\begin{center}
\begin{tabular}{cc}
\includegraphics[scale=0.1]{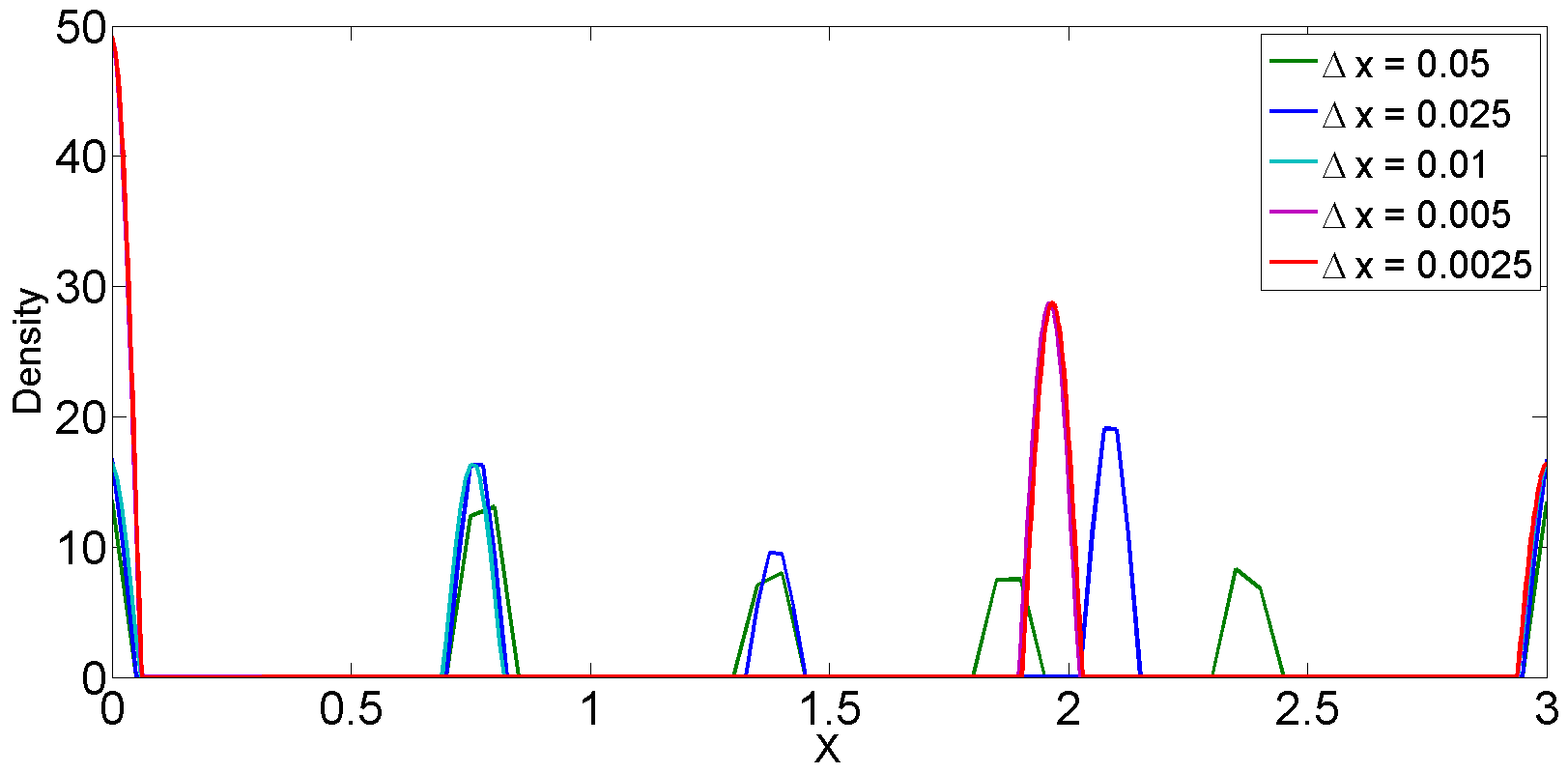}&\includegraphics[scale=0.1]{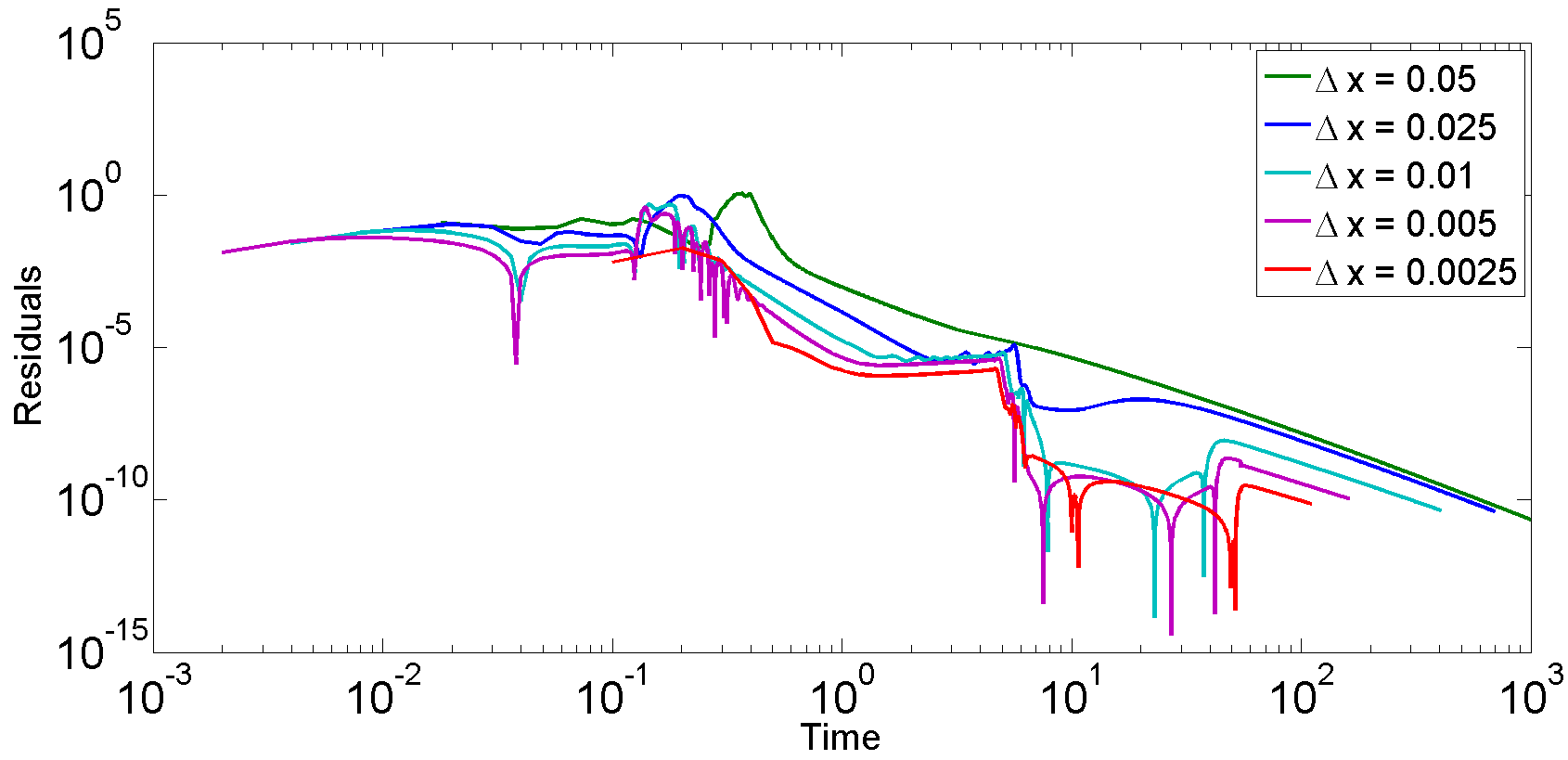}\\
\end{tabular}
\end{center}
\caption{Mesh refinement in the case $\gamma=2$. Density distribution  as a function of space (on the left) and residuals of the density  as a function of time in a log-log scale  (on the right) for the model \eqref{eq:main_system} with $\chi=10$, $D=0.1$, $a=20$, $b=10$, $\alpha=\kappa=1$ and $\gamma=2$, approximated using the P-reconstruction with the implicit treatment of the damping term. We use  different space steps $\Delta x = \{5 \times 10^{-2},2.5 \times 10^{-2}, 10^{-2}, 5 \times 10^{-3},2.5 \times 10^{-3}\}$.}
\label{fig:meshRefinement_gamma2}
\end{figure}
\begin{figure}[htbp!]
\begin{center}
\begin{tabular}{cc}
\includegraphics[scale=0.1]{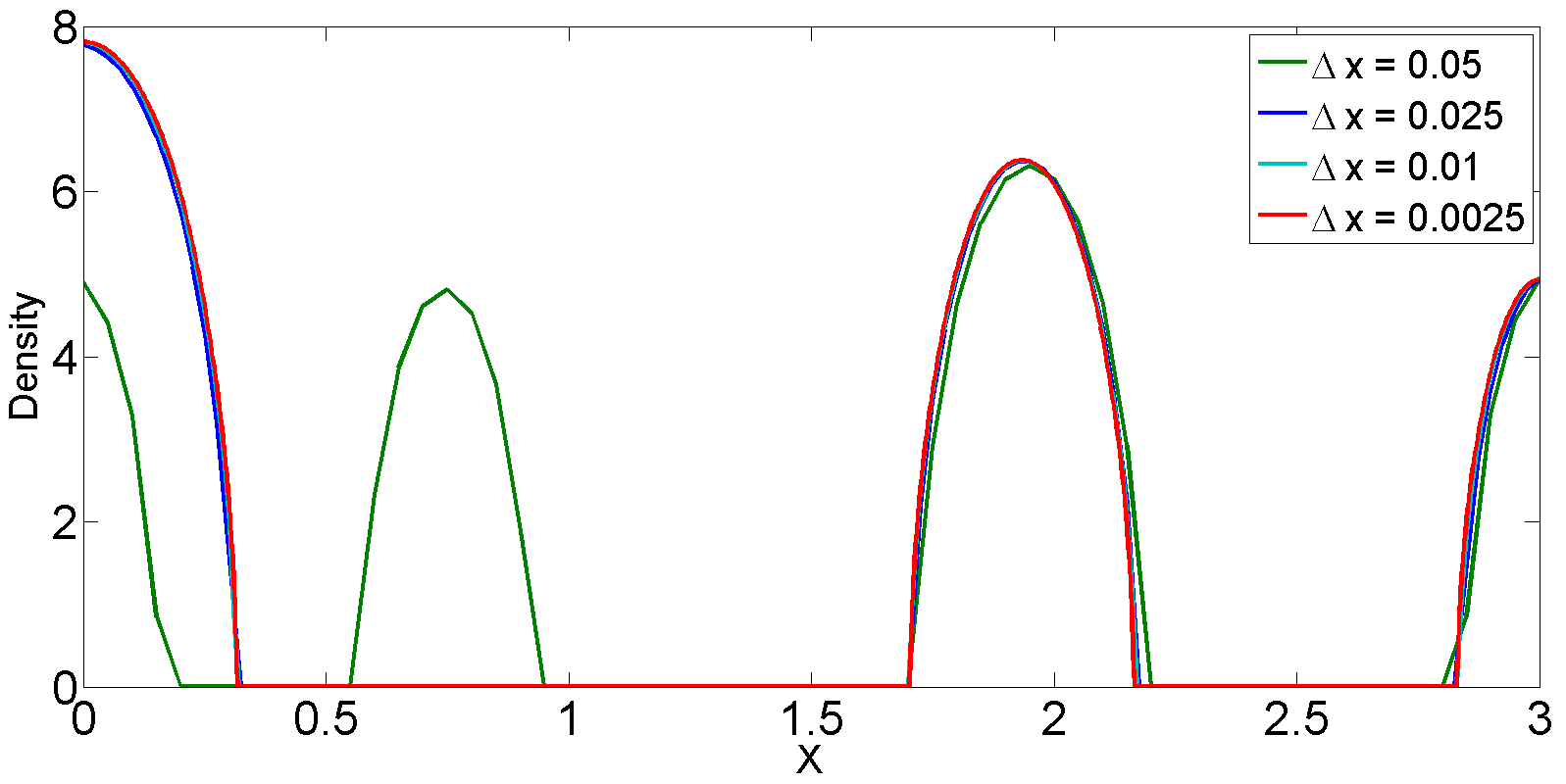}&\includegraphics[scale=0.1]{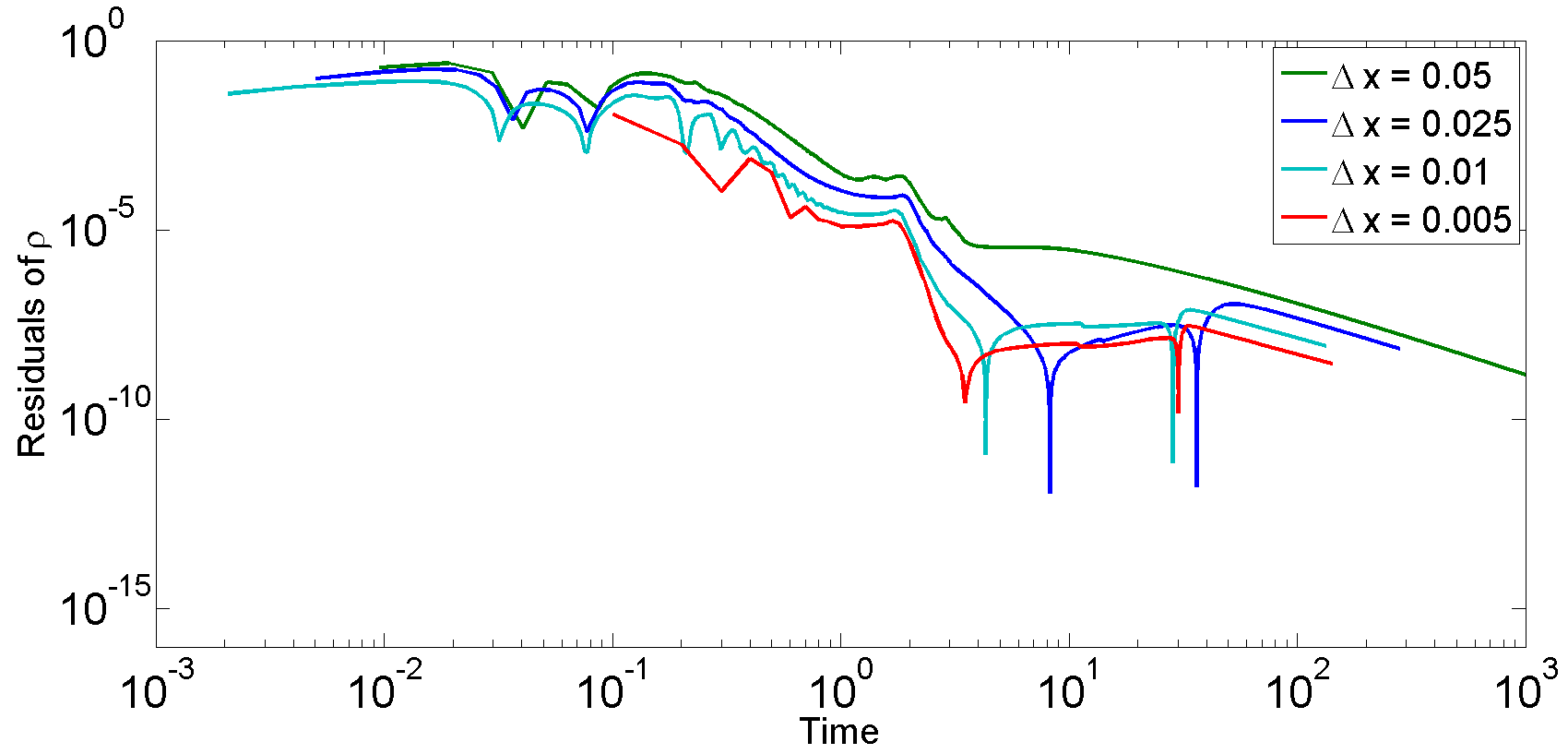}\\
\end{tabular}
\end{center}
\caption{Mesh refinement in the case $\gamma=3$. Density distribution  as a function of space  (on the left) and residuals of the density  as a function of time in a log-log scale  (on the right) for the model \eqref{eq:main_system} with  $\chi=10$, $D=0.1$, $a=20$, $b=10$, $\alpha=\kappa=1$ and $\gamma=3$, approximated using the P-reconstruction with the implicit treatment of the damping term. We use  different space steps $\Delta x = \{5 \times 10^{-2},2.5 \times 10^{-2}, 10^{-2}, 5 \times 10^{-3}\}$.}
\label{fig:meshRefinement_gamma3}
\end{figure}

We observe that for $\gamma=2$, the asymptotic solutions change significantly with the mesh refinement. More precisely, when we decrease the space step $\Delta x$, some of the neighboring bumps  merge together and the total number of regions with positive density becomes smaller. Finally the only correct solution appears to be the one for $\Delta x = 5 \times 10^{-3}$, which is stable under further refinements of the mesh. In the case $\gamma=3$,  the dependence on the mesh size is weaker and the difference is visible only for a very coarse mesh $\Delta x=5  \times 10^{-2}$. Below $\Delta x=2.5 \times 10^{-2}$, the asymptotic state does not change with the mesh refinement. Moreover, we see that the structure of the solution for the mesh of size $\Delta x=5 \times 10^{-3}$ is the same for $\gamma=2$ and $\gamma=3$,  that is  to say three bumps.

\section{Numerical comparison of the asymptotic behavior of the solutions to the quasilinear hyperbolic   and  the degenerate parabolic systems}\label{NumComp}

As explained before, on bounded domains with no-flux boundary conditions, stationary solutions for quasilinear hyperbolic system \eqref{eq:main_system} and parabolic system \eqref{eq:KellerSegel} are the same. A complete description of such stationary solutions has been recalled in Section \ref{recall}. Since we observed in \cite{NRT} that some of these stationary solutions with several bumps are asymptotic states of system  \eqref{eq:main_system}, we expect these solutions to be also  asymptotic solutions of the parabolic  model \eqref{eq:KellerSegel}. This guess  is motivated by the convergence of the solutions to the hyperbolic problem towards a parabolic-elliptic model in the LTLD limit;  see \cite{Marcati} for the model \eqref{eq:main_system} without chemotaxis  on unbounded domain and  \cite{DiFrancesco_Donatelli} for the full model, still  on an unbounded domain.

Therefore, in this section, our goal is  to compare carefully the large time behavior of these two models. First, we analyze the asymptotic solutions for initial data given by stationary solutions, which are  computed explicitly in the case $\gamma=2$. Notice that, due to numerical errors, the solutions are no longer stationary. In the case of  a non-symmetric stationary solution  composed of two lateral half bumps with different masses, we find that the asymptotic solutions of the two models are different. In the same way, for $\gamma=3$, we choose a generic initial datum  and we give examples of parameters for which the two systems stabilize asymptotically on different solutions; also  in that case, we show evidences of the appearance and  disappearance of some metastable patterns. 

Before the analysis of the asymptotic behavior of the solutions,  we briefly present an accurate scheme for the parabolic model \eqref{eq:KellerSegel} based on a relaxation technique.


\subsection{Numerical discretization of  degenerate parabolic model}\label{KSnum}

Now we describe the numerical scheme we use for the Keller-Segel type model \eqref{eq:KellerSegel}. As in the case of the hyperbolic model, the linear reaction-diffusion equation for the chemical concentration $\phi$  is solved using the second order centered finite differences method in space and a classical explicit-implicit Crank-Nicholson integration in time. At this point,  we focus only on the equation for the time evolution of the density 
\begin{equation}\label{eq:fv_chemotaxis}
\rho_{t}=\left(P(\rho)_{x}-\chi(\rho)\phi_{x}\right)_{x},
\end{equation}
where $\chi(\rho) = \chi\rho$. 

Notice that, due to the presence of vacuum,  the classical  second order centered explicit scheme
\begin{eqnarray*}
\rho_{i}^{n+1}&=&\rho_{i}^{n}+\frac{\Delta t}{\Delta x^2}\left(P_{i+1}^{n}-2P_{i}^{n}+P_{i-1}^{n}\right)\\
\nonumber&&\qquad-\frac{\chi\Delta t}{2\Delta x^2}\left[(\rho_{i}^{n}+\rho_{i+1}^{n})(\phi_{i+1}^{n}-\phi_{i}^{n}) -(\rho_{i-1}^{n}+\rho_{i}^{n})(\phi_{i}^{n}-\phi_{i-1}^{n})\right],
\end{eqnarray*}
where $\ds P_{i}^{n}=P(\rho_{i}^{n})$, fails to preserve the non negativity of the density \cite{BOP}. This is the reason why we consider the relaxation technique, based on the diffusive BGK approximation, introduced in \cite{ANT}. The advantage of this approach lies in a suitable modification of the diffusion term, split  into the linear and nonlinear parts, which guarantees the stability.  

Using the BGK approximation with two velocities  $\lambda_{1}=-\lambda_{2}=\lambda$, we are reduced to discretize linear transport problems, which is done thanks to  upwind method, see \cite{ANT} for more details, and we obtain the following scheme in the finite volume framework
\begin{equation*}
\rho_{i}^{n+1}=\rho_{i}^{n}+\frac{\Delta t}{\Delta x}\left(\mathcal{F}_{i+1/2}^n-\mathcal{F}^n_{i-1/2}\right),
\end{equation*}
with
\begin{equation*}
\mathcal{F}_{i+1/2}^{n} = \mathcal{F}_{i+1/2}^{n,\textrm{d1}}+\mathcal{F}_{i+1/2}^{n,\textrm{d2}}+\mathcal{F}_{i+1/2}^{n,\textrm{c}}.
\end{equation*}
Here the numerical flux is decomposed in three distinct parts, a nonlinear diffusive part with pressure, a linear diffusive part on the density and an advection part, taking into account the chemotactic term~:
\begin{equation*}
\left\{\begin{array}{l}
\ds{\mathcal{F}_{i+1/2}^{n,\textrm{d1}} = \mathcal{F}^{\textrm{diff1}}(P_{i}^n,P_{i+1}^n) = \left(\frac{1}{\sqrt{2}\Delta x}-\frac{\lambda}{2\theta^2}\right) \left(P_{i+1}^n-P_{i}^n\right)},\\
\ds{\mathcal{F}_{i+1/2}^{n,\textrm{d2}} = \mathcal{F}^{\textrm{diff2}}(\rho_{i}^n,\rho_{i+1}^n) = \frac{\lambda}{2}\left(\rho_{i+1}^n-\rho_{i}^n\right)},\\
\ds{\mathcal{F}_{i+1/2}^{n,\textrm{c}} = \mathcal{F}^{\textrm{chem}}(A_{i}^n,A_{i+1}^n) = \frac{1}{2}\left(A_{i}^n+A_{i+1}^n\right)},
\end{array}\right.
\end{equation*}
where $\ds  A_{i}^n = \chi\rho_{i}^n\frac{\phi_{i+1}^n-\phi_{i-1}^n}{2\Delta x}$ is a discretization of the term $\ds \chi\rho\phi_{x}$.
The following choice of the two remaining  parameters  $\theta$ and $\lambda$, 
\begin{equation*}
\theta=\max_{\rho}\sqrt{\frac{P'(\rho)}{1-\beta}},\qquad \lambda=\frac{\chi\max_{\phi}\phi_{x}}{\beta},
\end{equation*}
with $\beta \in ]0,1[$, guarantees the monotonicity of  the BGK approximation, while the stability is ensured by the following  CFL condition (see again \cite{ANT} for the details of the proofs):
\begin{equation*}
\Delta t\leq\min\left\{\frac{\Delta x}{\lambda},\frac{\Delta  x^2}{2\theta^2}\right\}.
\end{equation*}
In practice, we will take $\beta=0.95$.
 The splitting of the diffusion term into a linear part and a nonlinear part  gives  the consistency for $\Delta x,\Delta t\rightarrow 0$ and the presence of the artificial viscosity  preserves the non negativity of the density at vacuum.

In the following sections, we compare the asymptotic behavior of the quasilinear hyperbolic model \eqref{eq:main_system} and the  asymptotic behavior of the parabolic Keller-Segel model \eqref{eq:KellerSegel} using the previously described numerical schemes.

\subsection{Case of two non-symmetric bumps as  initial datum}\label{NonSym}
We consider the interval $L=4$ and an initial datum composed of two lateral bumps with different masses. The solution is constructed  using the formulas of Section~\ref{sec:lateral_bump} with L=1 for each bump and $M=1$ for the left bump and $M=3$ for the right bump. The length of the support of each bump is the same since  it is independent of  the mass $M$. In Figure~\ref{fig:twoExactBumpsResiduals}, we display the asymptotic solutions at time $t=150$ for the hyperbolic and parabolic models (on the left) and  the residuals of the density as a function of time in a log-log scale (on the right). For the hyperbolic model, the initial density  is stable, while in the case of the parabolic model,  the smallest bump  merges with the largest bump around time $t=10$. 
\begin{figure}[htbp!]
\begin{center}
\begin{tabular}{cc}
\includegraphics[scale=0.1]{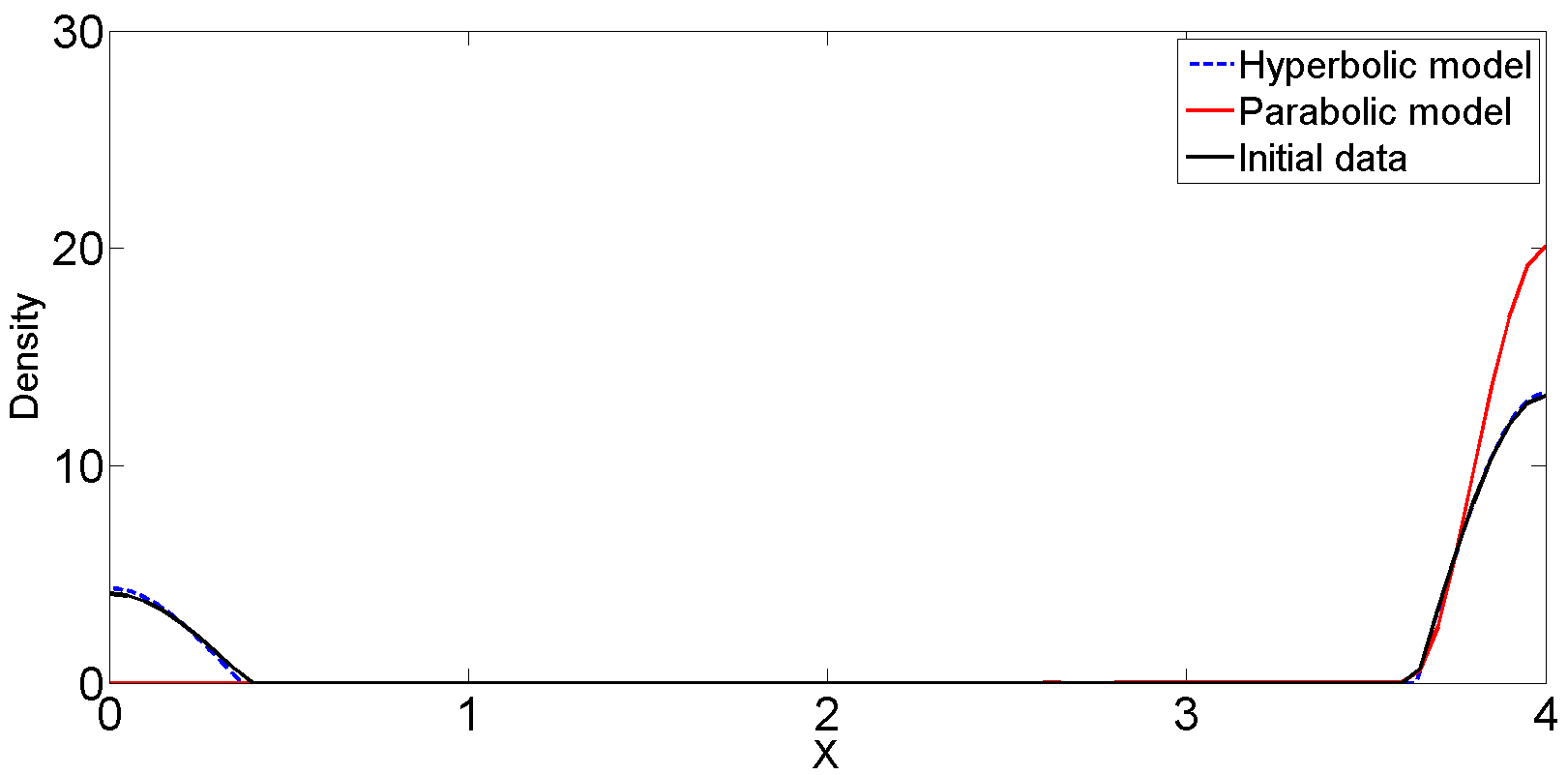}&
\includegraphics[scale=0.1]{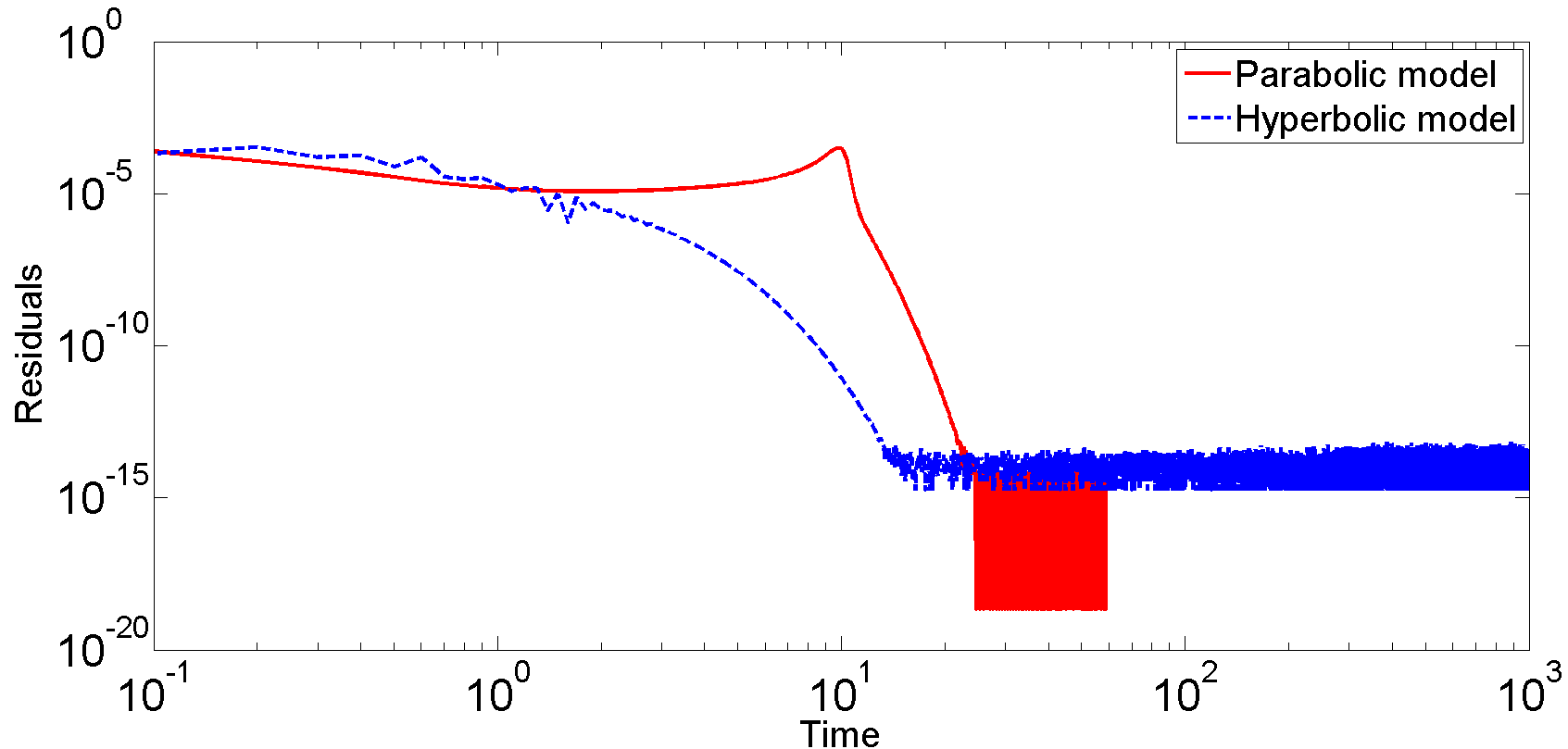}
\end{tabular}
\end{center}
\caption{Density distributions as a function of space at the asymptotic states (on the left) and the residuals of the density as a function of time in the log-log scale (on the right) for the quasilinear hyperbolic model \eqref{eq:main_system} and the parabolic Keller-Segel system \eqref{eq:KellerSegel} with $\gamma=2$, $\chi=50$, $D=a=b=\alpha=\kappa=1$ and the exact, non-symmetric steady solution composed of two lateral bumps with different masses as initial datum.}
\label{fig:twoExactBumpsResiduals}
\end{figure}

Remark that, in that case, we exhibit a set of parameters and initial data for which the two models, hyperbolic and parabolic, have different asymptotic behaviors. Namely, in that case, the diffusivity of the parabolic model does not permit to keep distinct bumps and leads to the emergence of a single bump structure.

\subsection{Generic initial data}\label{generic}

In the previous section, we observed a case of  different asymptotic states for the hyperbolic and parabolic models, that is to say asymptotic solutions composed of a different number of bumps. Now,  we compare the behavior of the two systems for a generic initial datum.  More precisely, we consider an interval of length $L=3$ with some parameters $\chi=10$, $D=0.1$, $a=20$, $b=10$, $\kappa=\alpha=1$ and $\gamma=3$; we take as an initial density the function $\rho_{0}(x)=1.5+\sin(4\pi|x-0.25L|)$ and the initial condition for the concentration $\phi$ is $\phi_{0}=0$. In Section~\ref{NumHyp},  we studied the dependence of the solutions of the hyperbolic model on the mesh refinement. We observed a high sensitivity to the space step  in the case $\gamma=2$. That is why, to avoid possible inaccurate results, we consider now only the case  $\gamma=3$, for which the acceptable grid size is much larger. In the following test, we perform simulations on a  mesh with space step $\Delta x=0.01$. The solution of the hyperbolic system was also verified on a finer mesh with $\Delta x=0.005$ and the density distribution did not change.

In Figure~\ref{fig:test2_gamma3}, we present the density  at different times for the quasilinear hyperbolic model and the parabolic system. In the case of the parabolic model,  we observe some  metastable patterns. This phenomenon has been already observed, but without vacuum, for the Keller-Segel type model with linear diffusion and logistic chemosensitive function, see \cite{HP2001}. For the parabolic model, we start with a smooth perturbation of an initial constant  solution, and a ''comb'' structure with several bumps appears and remains almost unchanged for a while, that is  to say between $t\sim2$ and $t\sim10$. Then,  a fast transition takes place and one of the interior bumps moves towards the boundary. This structure seems again to be frozen up to $t\sim300$ when another transition occurs and only one bump remains. Analyzing the residuals of the density in Figure~\ref{fig:test2_gamma3_residuals},  we note that the time between the subsequent transitions becomes larger as the number of bumps decreases. 

Moreover, the initial evolution of the hyperbolic model is much faster and the  interactions between bumps and the merging of some of them end between time $t=10$ and $t=20$.  In the end, one bump and two half  bumps remain and the solution stabilizes, which is confirmed by the residuals of the density.  Again, the hyperbolic and the parabolic asymptotic behaviors are different. 
\begin{figure}[htbp!]
\begin{center}
\begin{tabular}{cc}
\includegraphics[scale=0.1]{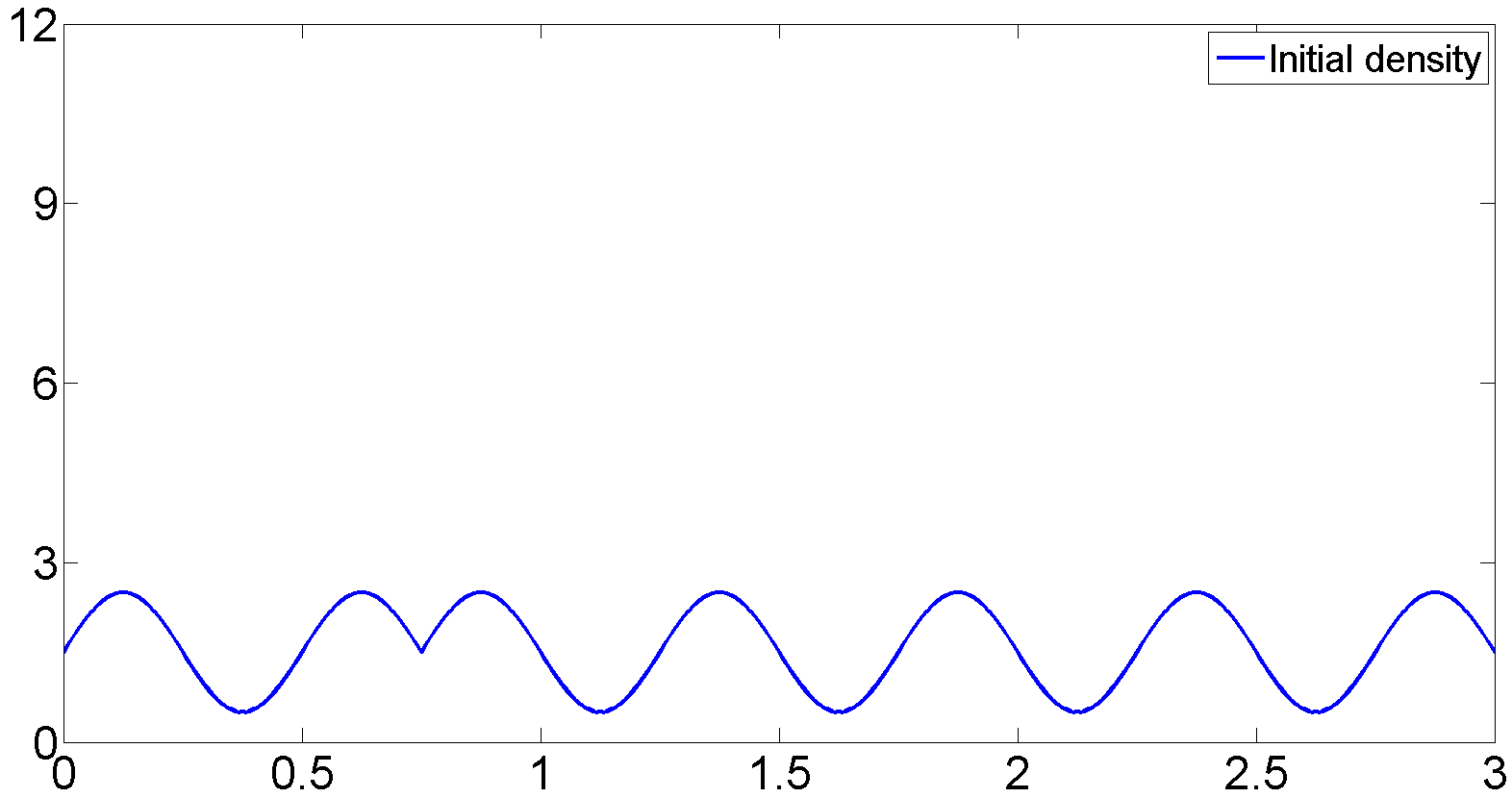}&\includegraphics[scale=0.1]{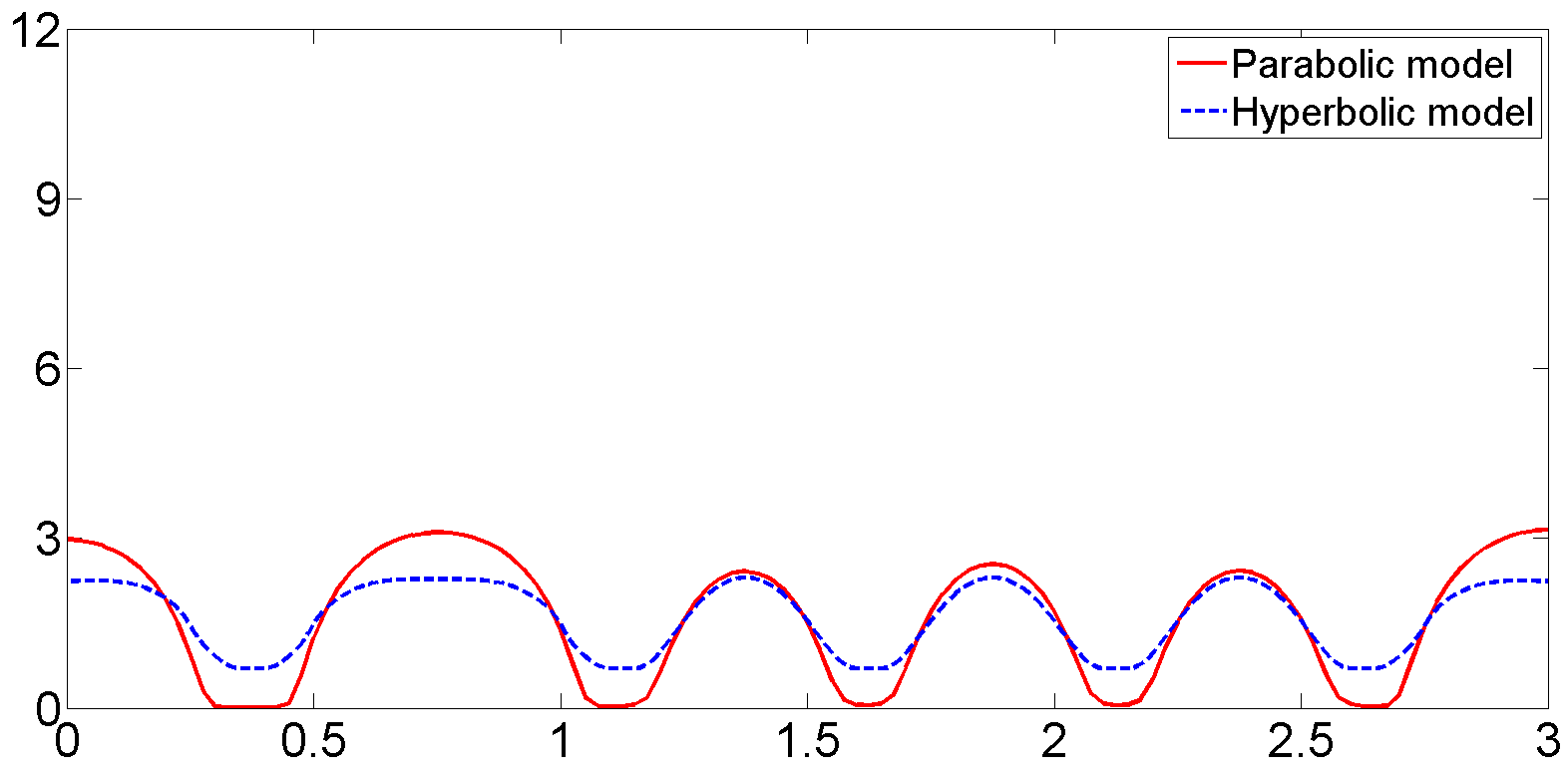}\\
t = 0&t = 0.1\\
\includegraphics[scale=0.1]{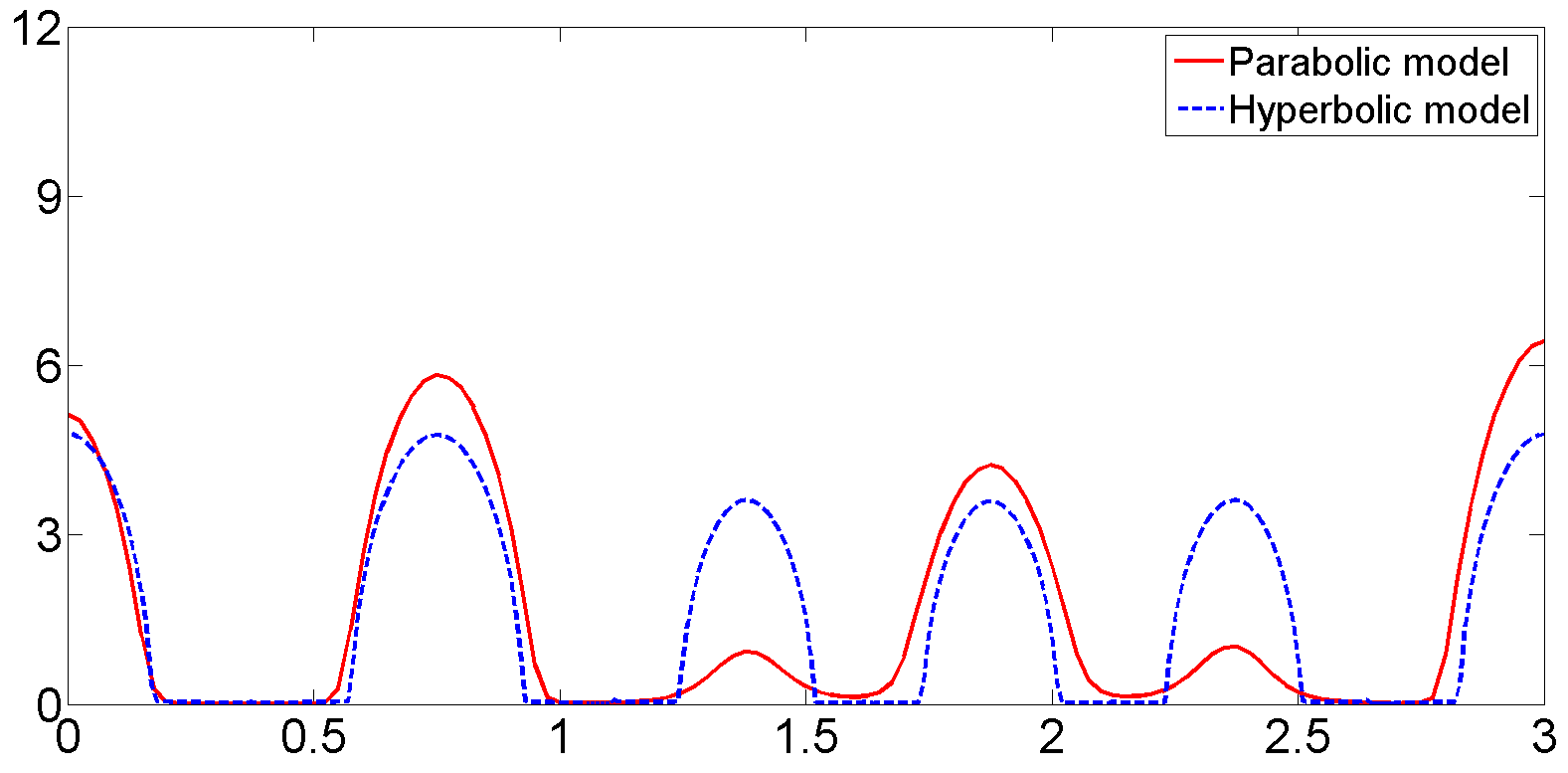}&\includegraphics[scale=0.1]{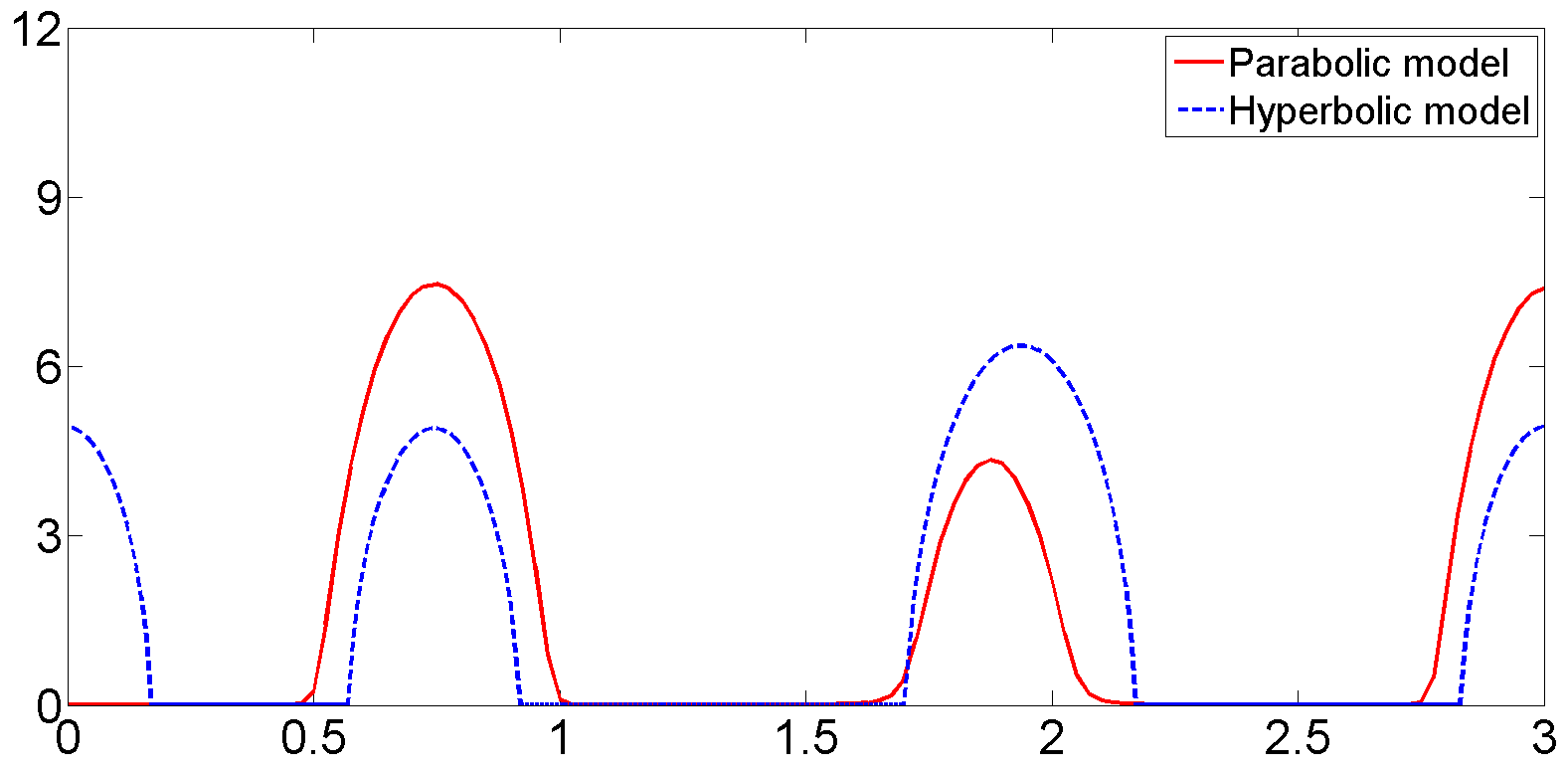}\\
t = 0.5&t = 3\\
\includegraphics[scale=0.1]{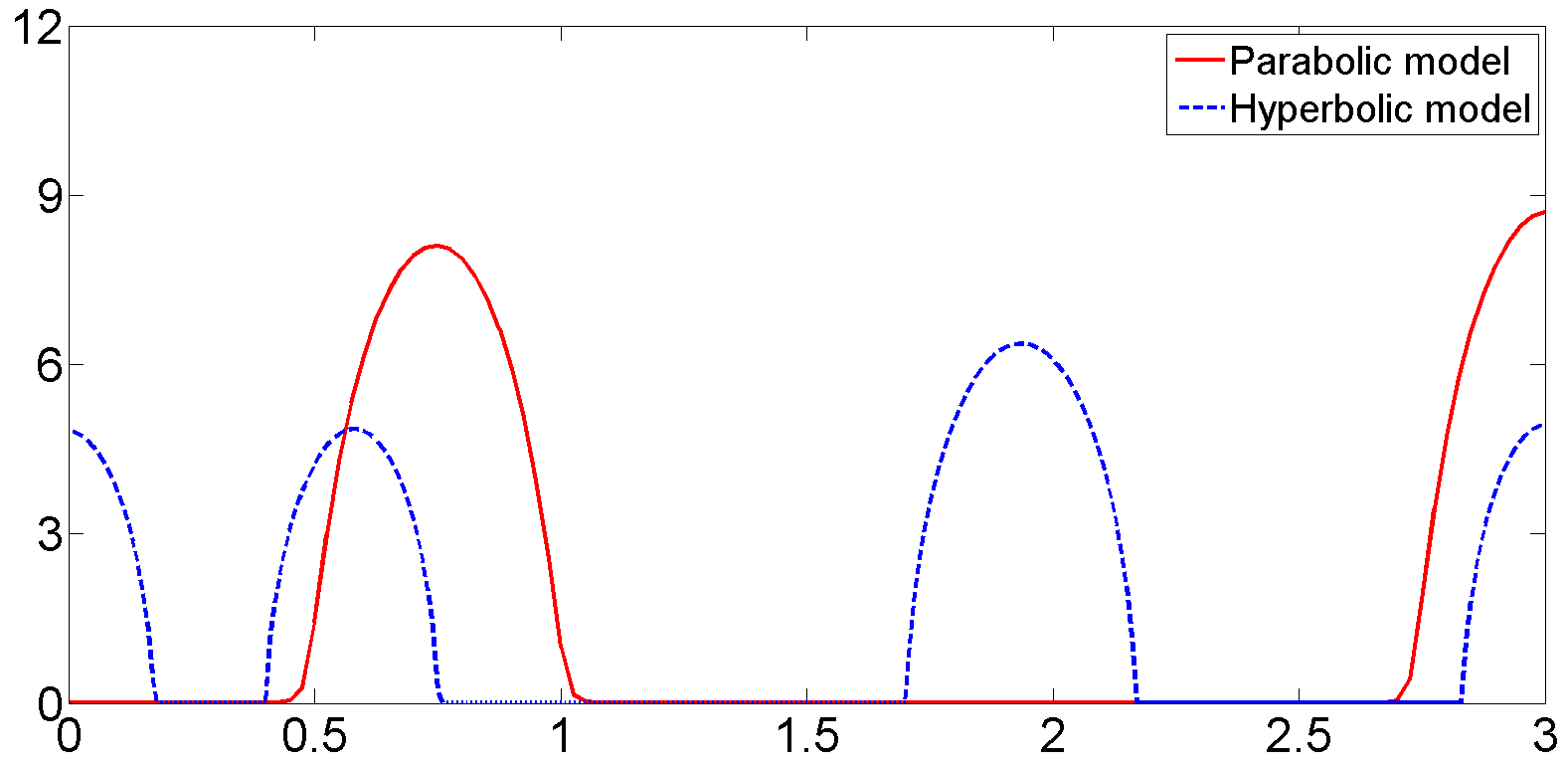}&\includegraphics[scale=0.1]{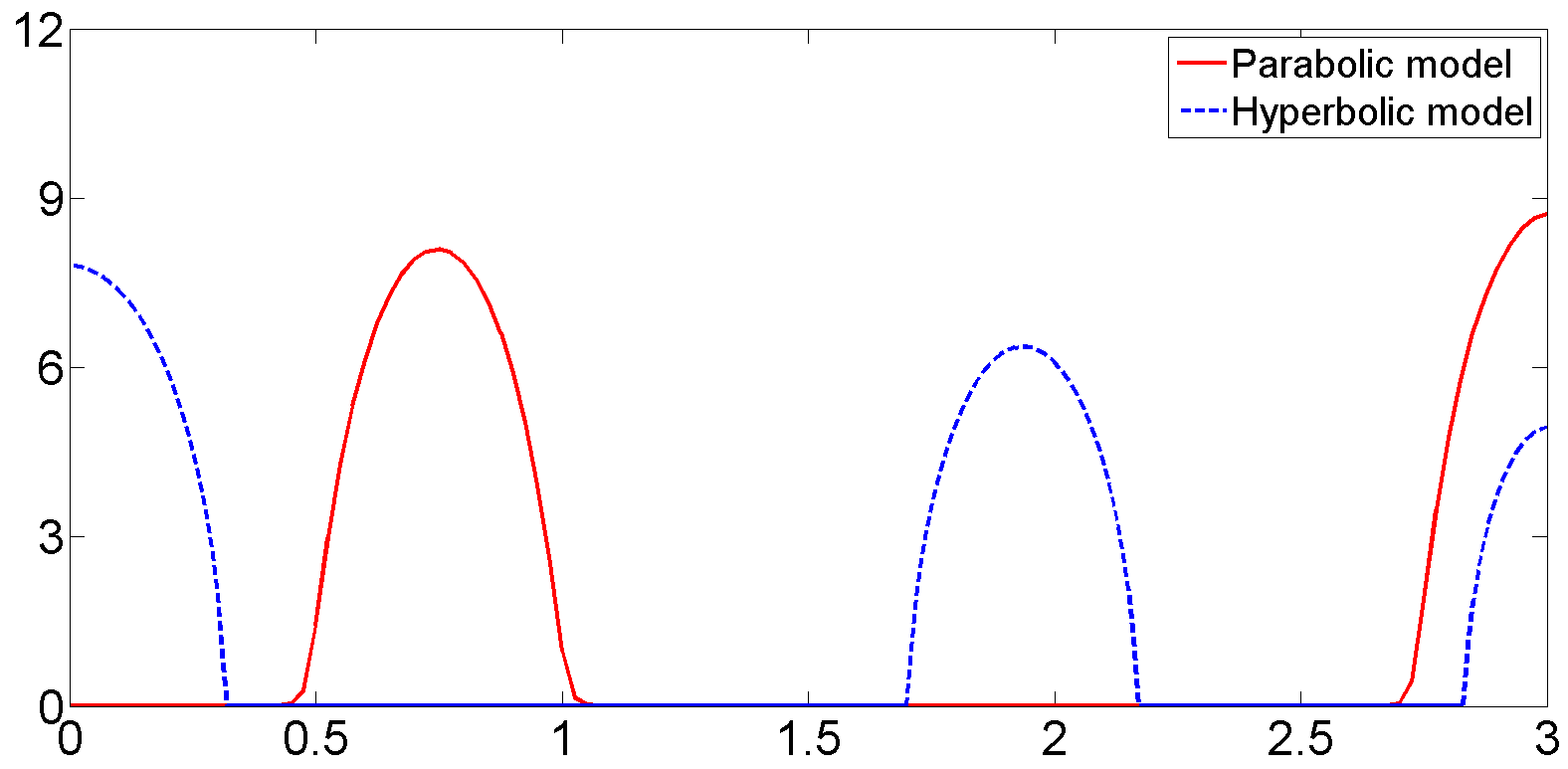}\\
t = 10&t = 20\\
\includegraphics[scale=0.1]{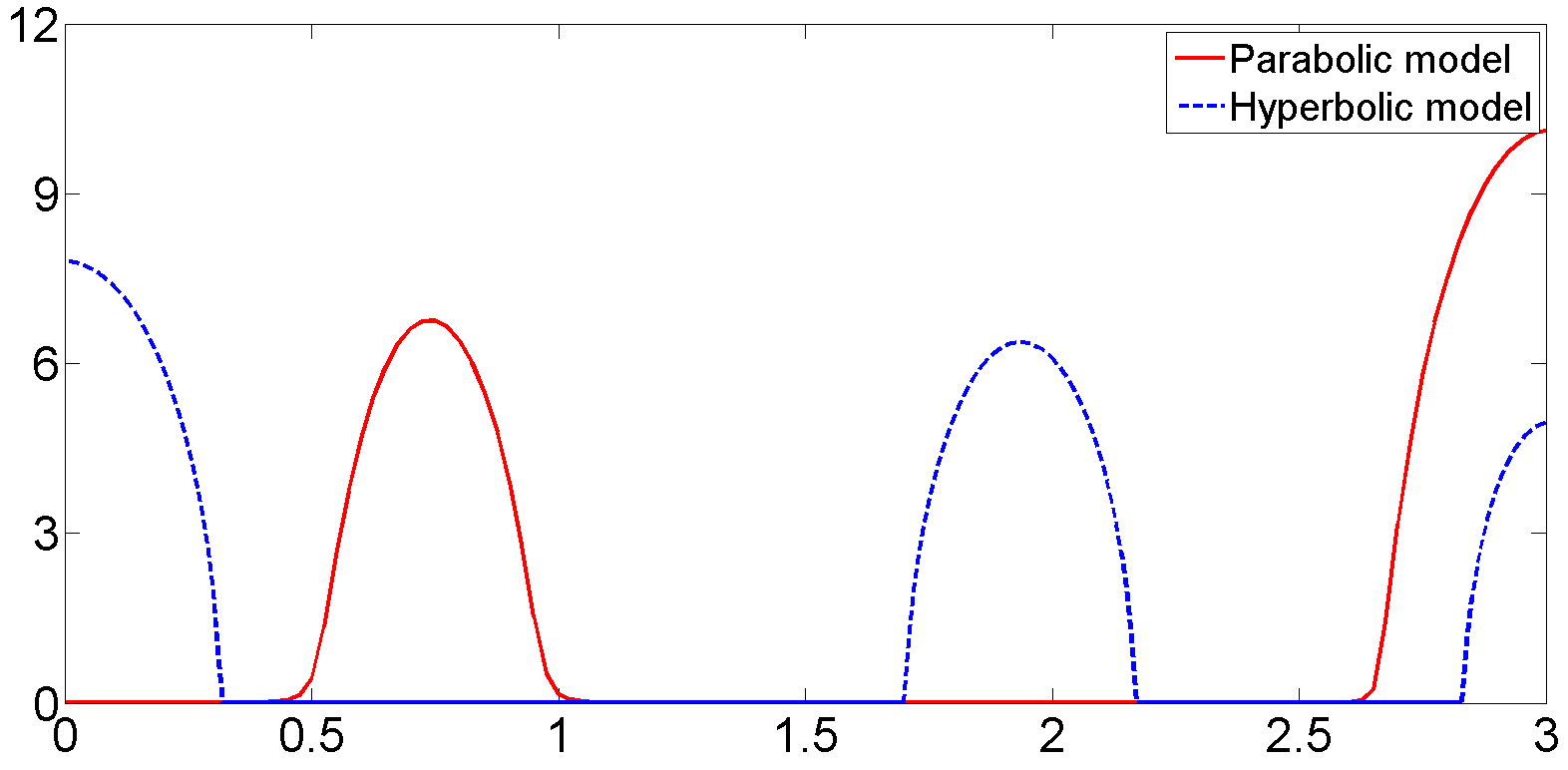}&\includegraphics[scale=0.1]{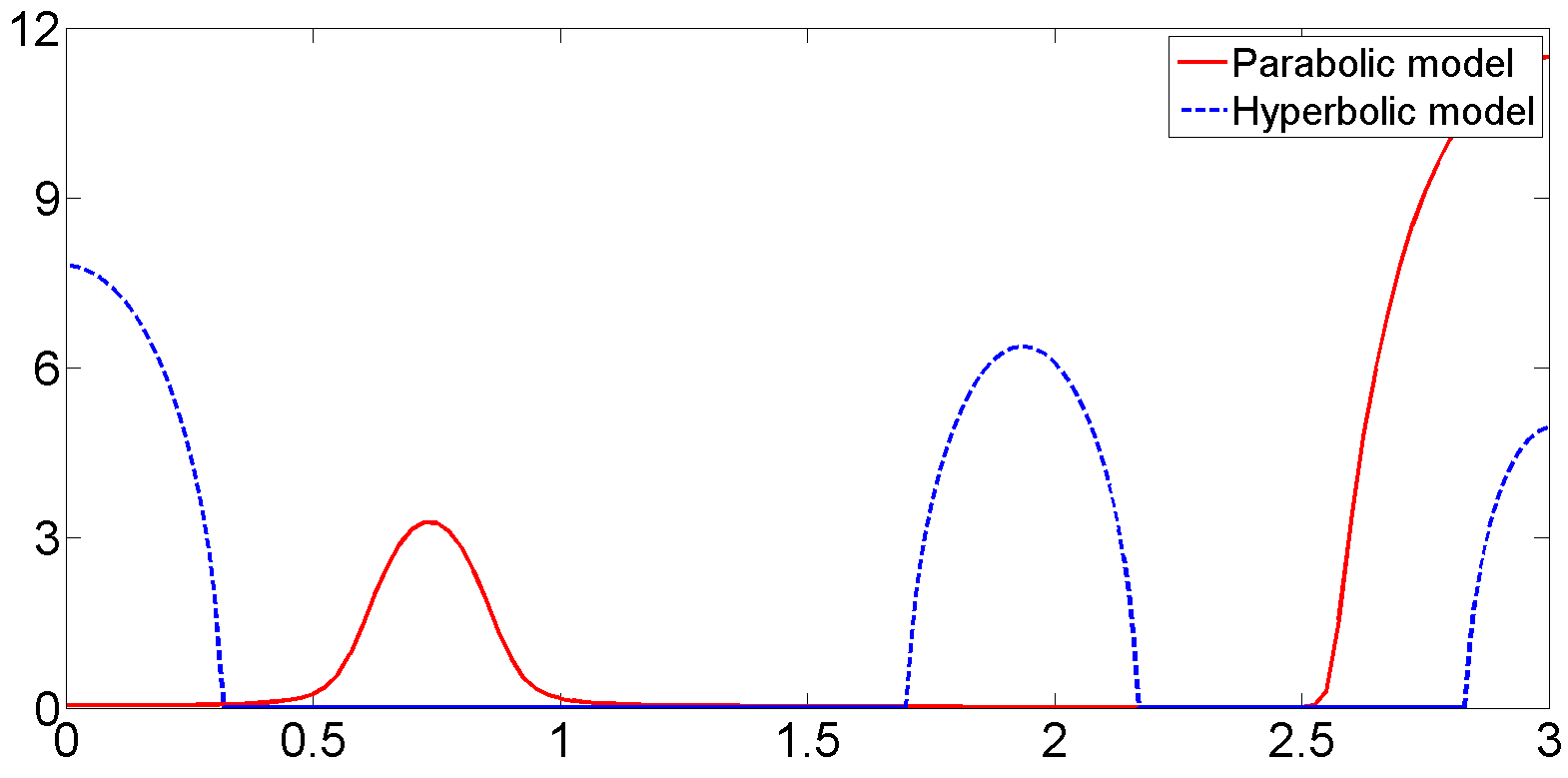}\\
t = 300&t = 331\\
\includegraphics[scale=0.1]{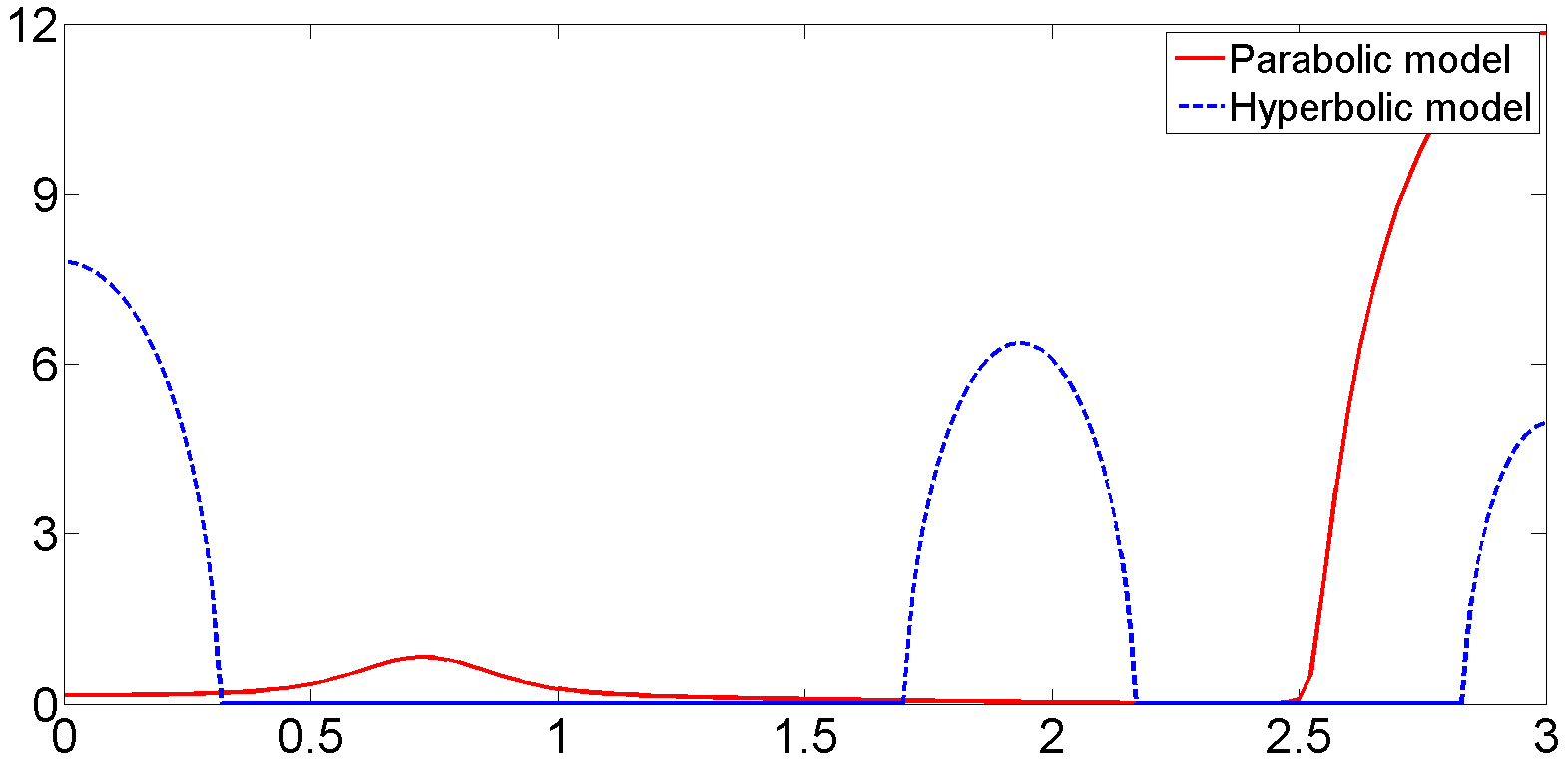}&\includegraphics[scale=0.1]{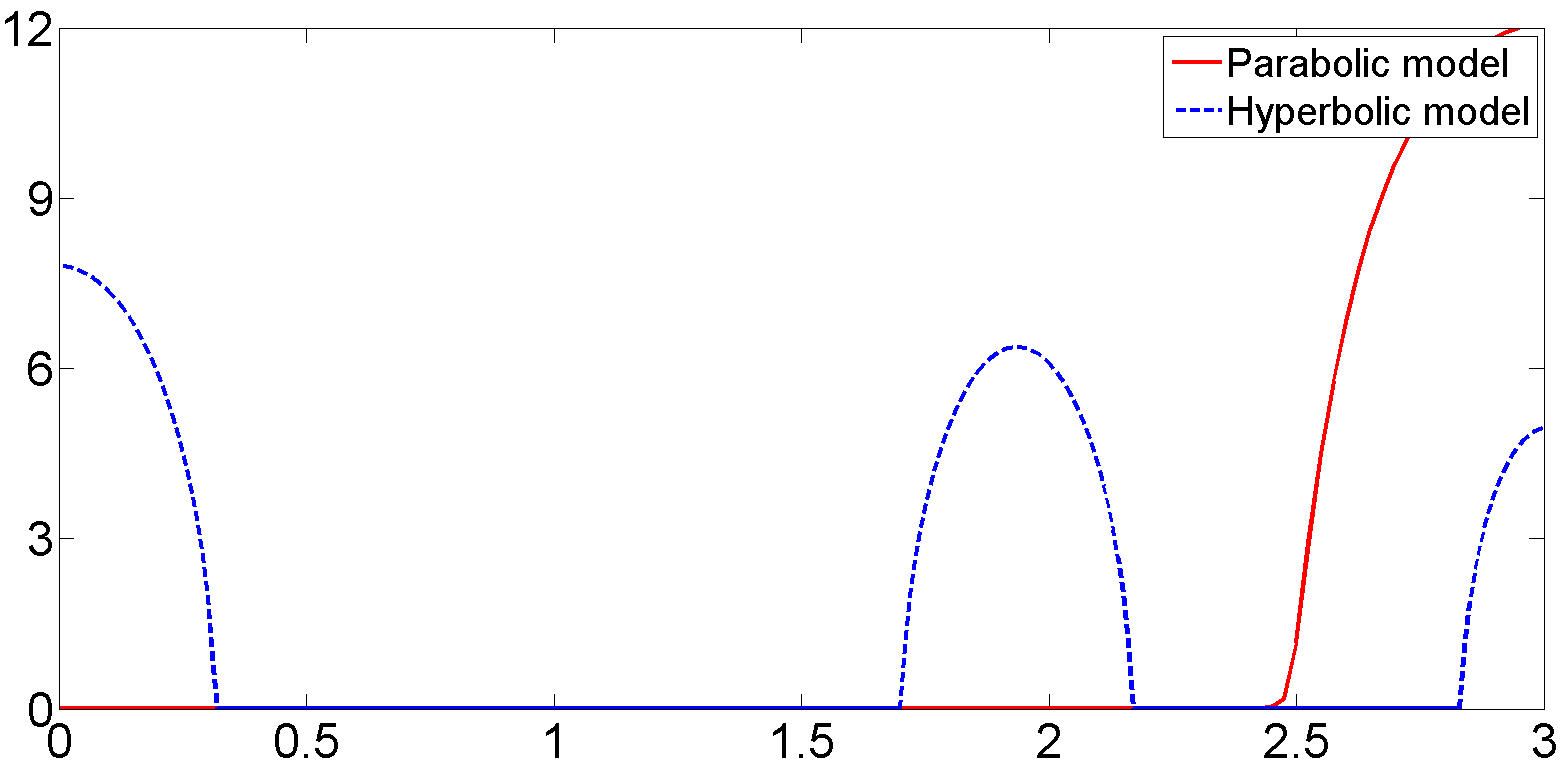}\\
t = 332&t = 333\\
\end{tabular}
\caption{Density $\rho$ as  a function of space at different times for the quasilinear hyperbolic model of chemotaxis \eqref{eq:main_system} and for  the Keller-Segel system \eqref{eq:KellerSegel} with $\chi=10$, $D=0.1$, $a=20$, $b=10$, $\kappa=\alpha=1$ and $\gamma=3$. The initial data are $\rho_{0}(x)=1.5+\sin(4\pi|x-0.25L|)$ and  $\phi_{0}(x)=0$.}
\label{fig:test2_gamma3}
\end{center}
\end{figure}

\begin{figure}[htbp!]
\begin{center}
\begin{tabular}{c}
\includegraphics[scale=0.17]{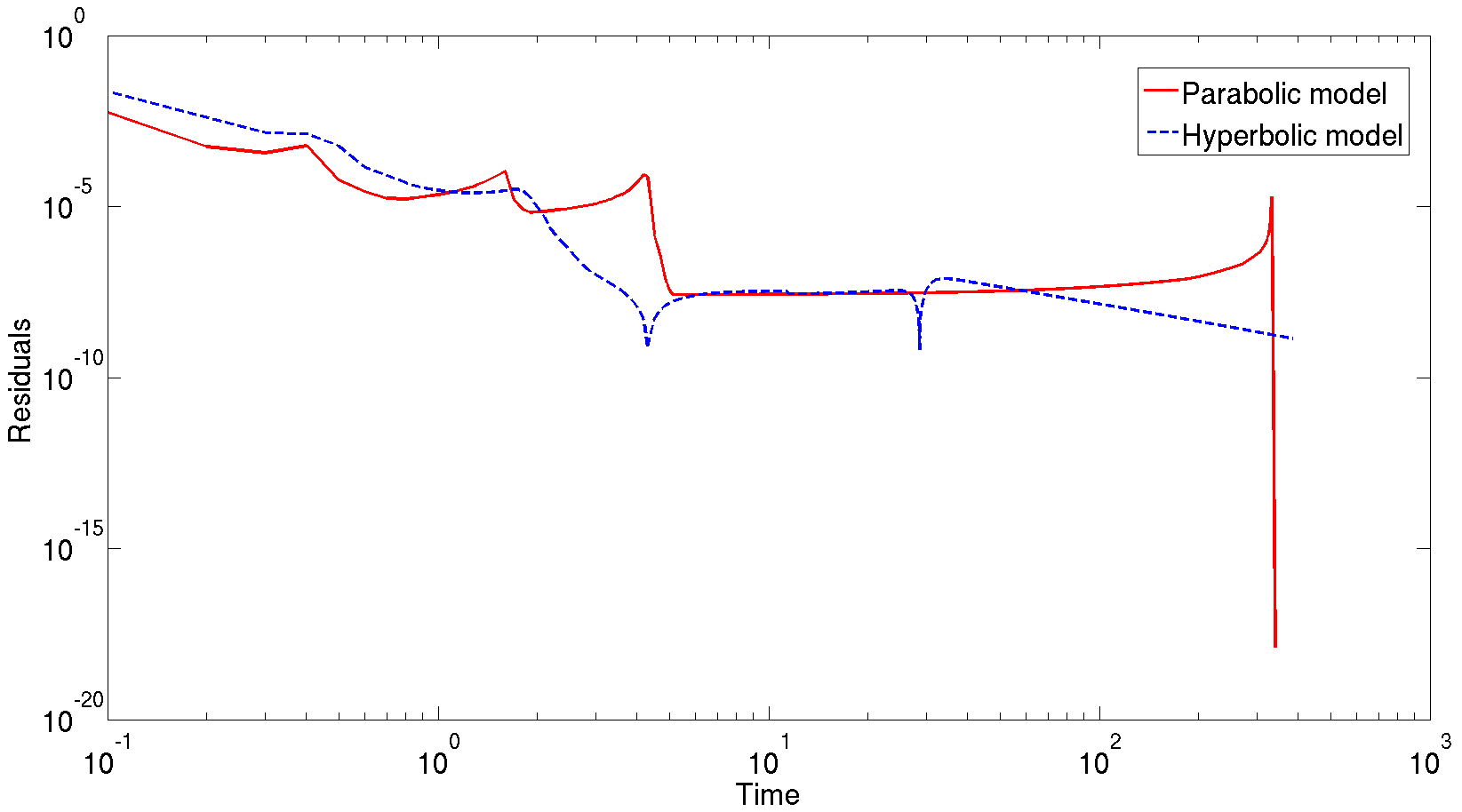}
\end{tabular}
\caption{Residuals of the density as a function of time in a log-log scale for the quasilinear hyperbolic model of chemotaxis \eqref{eq:main_system} and the Keller-Segel system \eqref{eq:KellerSegel}  with $\chi=10$, $D=0.1$, $a=20$, $b=10$, $\kappa=\alpha=1$ and $\gamma=3$. The initial data are $\rho_{0}(x)=1.5+\sin(4\pi|x-0.25L|)$ and  $\phi_{0}(x)=0$.}
\label{fig:test2_gamma3_residuals}
\end{center}
\end{figure}

{\bf Acknowledgement.} The authors thank Fran\c{c}ois Bouchut for some useful suggestions. This work has been partially supported by
the project PORAbruzzo and   by the ANR project MONUMENTALG, ANR-10-JCJC 0103. 

\section{Annex : numerical fluxes and definition \ref{thedef} }\label{annex}
Let us first  recall the  definition \ref{thedef}. A strongly consistent numerical flux $\ds \mathcal{F}$ satisfies the two following conditions~:
\begin{equation}\label{symmetric}
\left\{\begin{array}{l}
\textrm{if } \mathcal{F}^{\rho u}(r,0,R,0)=P(r), \textrm{ then } r=R;
\medskip
\\
 \textrm{if } \mathcal{F}^{\rho u}(r,0,R,0)=P(R), \textrm{ then } r=R.\\
\end{array}\right.
\end{equation}

We consider the following equations~:
\begin{equation}\label{theeq}
\mathcal{F}^{\rho u}(r_{i+1/2}^{n,-},0,r_{i+1/2}^{n,+},0)-\mathcal{F}^{\rho u}(r_{i-1/2}^{n,-},0,r_{i-1/2}^{n,+},0)= P\left(r_{i+1/2}^{n,-}\right)-P\left(r_{i-1/2}^{n,+}\right),
\end{equation}
computed at the beginning of the proof of  Theorem \ref{thethm}. In the following, we will prove that  conditions \eqref{symmetric}
are necessary and sufficient conditions to ensure  that the equalities  $r_{i-1/2}^{n,-}=r_{i-1/2}^{n,+}$ for all $i$ are  the unique solutions of equations \eqref{theeq}. We will also show that  the following classical fluxes~: HLL, HLL-Roe and Suliciu relaxation flux adapted  to vacuum are indeed strongly consistent fluxes.

Remark that  conditions  \eqref{symmetric} have already been derived as  necessary conditions on the flux  in \cite{BOP}, where a sufficient condition on the flux is also given to ensure the uniqueness  property for the solutions of equations \eqref{theeq}.  Since  we are dealing here with the bounded domain case with boundary conditions \eqref{boundary_conditions}, our computations are slightly different from the ones of \cite{BOP} and  we  are able to prove that the conditions \eqref{symmetric}   are also sufficient conditions. 

Indeed, considering equations \eqref{theeq} for all $i$, using that $r_{1/2}^{n,-}=r_{1/2}^{n,+}$ thanks to boundary conditions and using  that the flux $\mathcal{F}$ is consistent, a straightforward  induction implies that 
\begin{equation}\label{theeq2}
\mathcal{F}^{\rho u}(r_{i+1/2}^{n,-},0,r_{i+1/2}^{n,+},0)= P\left(r_{i+1/2}^{n,-}\right), \textrm{ for all } i.
\end{equation}
Using condition \eqref{symmetric}, we obtain that  $r_{i-1/2}^{n,-}=r_{i-1/2}^{n,+}$ for all $i$. Therefore, condition \eqref{symmetric} is a necessary and sufficient condition to guarantee that  the equalities  $r_{i-1/2}^{n,-}=r_{i-1/2}^{n,+}$ for all $i$ are  unique solutions of equations \eqref{theeq}.

Now, let us show that HLL, HLL-Roe and Suliciu with vacuum fluxes satisfy conditions  \eqref{symmetric}. We assume in the following that the functions $P$ and  $P'$ are increasing, as satisfied by the pressure  \eqref{eq:pressure_law} we consider here.

\textbf{HLL flux. } The definition of HLL flux is given at eq. (2.111)  in Bouchut's book \cite{Bouchut_book} and we can compute
\begin{equation*}
\mathcal{F}^{\rho u}(r,0,R,0)=\frac{c_{2}P(r)-c_{1}P(R)}{c_{2}-c_{1}},
\end{equation*}
with $\ds c_{1}=\min(-\sqrt{P'(r)},-\sqrt{P'(R)} )$ and $\ds c_{2}=\max(\sqrt{P'(r)},\sqrt{P'(R)} )$, that is to say
\begin{equation*}
\mathcal{F}^{\rho u}(r,0,R,0)=\frac{P(r)+P(R)}{2},
\end{equation*}
which satisfies clearly conditions \eqref{symmetric}.

\textbf{HLL-Roe   flux.} In \cite{Einfeldt}, we can find a version of the HLL flux adapted to vacuum. In that case,  
\begin{equation*}
\mathcal{F}^{\rho u}(r,0,R,0)=\frac{c_{2}P(r)-c_{1}P(R)}{c_{2}-c_{1}},
\end{equation*}
with  $\ds c_{1}=\min(-\sqrt{P'(r)},-\bar c )$ and $\ds c_{2}=\max(\bar c,\sqrt{P'(R)} )$, where $\ds \bar c=\sqrt{\frac{\sqrt R P'(R)+\sqrt rP'(r)}{\sqrt R+\sqrt r}}$, that is to say
\begin{equation*}
\mathcal{F}^{\rho u}(r,0,R,0)=\left \{ 
\begin{array}{ll}
\ds\frac{\sqrt{P'(R)} P(r)+\bar cP(R)}{\sqrt{P'(R)} +\bar c} ,& \textrm{ if } R>r, \\
\ds\frac{\bar cP(r)+\sqrt{P'(r)}P(R)}{\bar c+\sqrt{P'(r)}}, & \textrm{ if } r>R. 
\end{array}\right.
\end{equation*}
From this expression, we conclude easily that HLL-Roe flux satisfies conditions \eqref{symmetric}.

\textbf{Suliciu flux adapted to vacuum.} Now, we consider the Suliciu relaxation flux adapted to vacuum, which expression can be found in \cite{Bouchut_book} at equations (2.133)-(2.136).

 If $0<r<R$, a standard computation leads to 
\begin{equation*}
\mathcal{F}^{\rho u}(r,0,R,0)=\frac{c_{2}P(r)+c_{1}P(R)}{c_{1}+c_{2}}+\frac{Rc_{2}}{c_{1}+c_{2}}\times \frac{(P(r)-P(R))^2}{c_{2}(c_{1}+c_{2})+R(P(R)-P(r))},
\end{equation*}
with $\ds c_{1}=r\sqrt{P'(r)}+\alpha r \left(\frac{P(R)-P(r)}{R\sqrt{P'(R)}}\right)>0$ and $\ds c_{2}=R\sqrt{P'(R)}>0$.

If $r>R>0$, we obtain a similar formula, namely
\begin{equation*}
\mathcal{F}^{\rho u}(r,0,R,0)=\frac{c_{2}P(r)+c_{1}P(R)}{c_{1}+c_{2}}+\frac{rc_{1}}{c_{1}+c_{2}}\times \frac{(P(r)-P(R))^2}{c_{1}(c_{1}+c_{2})+r(P(r)-P(R))},
\end{equation*}
with $\ds c_{1}=r\sqrt{P'(r)}>0$ and $\ds c_{2}=R\sqrt{P'(R)}+\alpha R \left(\frac{P(r)-P(R)}{r\sqrt{P'(r)}}\right)>0$.

Now, we consider the equation $\ds \mathcal{F}^{\rho u}(r,0,R,0)=P(r)$. On the one hand, in the case $r<R$,
\begin{equation*}
\begin{split}
0=&\mathcal{F}^{\rho u}(r,0,R,0)-P(r)\\
=&\frac{c_{1}(P(R)-P(r))}{c_{1}+c_{2}}+\frac{Rc_{2}}{c_{1}+c_{2}}\times \frac{(P(r)-P(R))^2}{c_{2}(c_{1}+c_{2})+R(P(R)-P(r))}.
\end{split}
\end{equation*}
Since the right-hand side of the last equation is the sum of two positive terms, it is straightforward that they are both null and that $P(R)=P(r)$, which leads to $R=r$.
On the other hand, in the case $r>R$,
\begin{equation*}
\begin{split}
0=&\mathcal{F}^{\rho u}(r,0,R,0)-P(r)\\
=&\frac{c_{1}(P(R)-P(r))}{c_{1}+c_{2}}+\frac{rc_{1}}{c_{1}+c_{2}}\times \frac{(P(r)-P(R))^2}{c_{1}(c_{1}+c_{2})+r(P(r)-P(R))}.
\end{split}
\end{equation*}
This equation can be simplified as~:
\begin{equation*}
c_{1}(c_{1}+c_{2})=0 \textrm{ or }P(r)-P(R)=0.
\end{equation*}
Since the first equality  is impossible, we conclude that $r=R$.

Notice that the equation $\ds \mathcal{F}^{\rho u}(r,0,R,0)=P(R)$ can be treated in a similar way. Therefore, we have proved that the Suliciu relaxation flux satisfies also the conditions \eqref{symmetric}.

\bibliographystyle{plain}

\end{document}